\documentclass[11pt,twoside]{article}
\usepackage{etex}
\usepackage[dvipsnames]{xcolor} 
\usepackage{color}
\usepackage{tcolorbox}
\usepackage{booktabs}
\usepackage{amsthm}
\usepackage{amsfonts,amssymb,amsxtra,url,float} 
\allowdisplaybreaks[4] 
\usepackage[colorlinks,
  linkcolor=magenta, %
  anchorcolor=Periwinkle,
  citecolor=red,
  urlcolor=blue
  ]{hyperref} 
\usepackage{enumitem}
\usepackage{geometry} 
\usepackage{rotating} 
\usepackage{lscape} 
\usepackage{multirow}
\usepackage{graphicx} 
\usepackage{subfigure} 
\usepackage{tikz}
\usepackage{pgfplots}
\usepackage{tikz-3dplot}
\usetikzlibrary{patterns}
\usetikzlibrary{3d,calc}
\usetikzlibrary{decorations.pathreplacing,decorations.markings}
 \tikzset{
  on each segment/.style={
    decorate,
    decoration={
      show path construction,
      moveto code={},
      lineto code={
        \path [#1]
        (\tikzinputsegmentfirst) -- (\tikzinputsegmentlast);
      },
      curveto code={
        \path [#1] (\tikzinputsegmentfirst)
        .. controls
        (\tikzinputsegmentsupporta) and (\tikzinputsegmentsupportb)
        ..
        (\tikzinputsegmentlast);
      },
      closepath code={
        \path [#1]
        (\tikzinputsegmentfirst) -- (\tikzinputsegmentlast);
      },
    },
  },
  mid arrow/.style={postaction={decorate,decoration={
        markings,
        mark=at position 0.6 with {\arrow[#1]{stealth}} 
      }}},
}
\usetikzlibrary{arrows}
\usetikzlibrary{trees}

\usetikzlibrary{matrix}
\usetikzlibrary{patterns}
\usetikzlibrary{shadings} 
\usepackage{fancyhdr} 
\def\headertitle{Kleisli convolution representations of power monoids}
\def\fstpage{1} 
\def\page{$\begin{matrix} {\color{white}0} \\ \thepage \end{matrix}$} 

\usepackage[all]{xy} 
\usepackage{dsfont} 
\usepackage{cite}
\usepackage{mathrsfs} 
\numberwithin{figure}{section}
\usepackage{marginnote} 
\usepackage{graphicx} 
\usepackage{multicol} 

\usepackage{enumitem}
\setenumerate[1]{itemsep=0pt,partopsep=0pt,parsep=\parskip,topsep=3pt}
\setitemize[1]{itemsep=0pt,partopsep=0pt,parsep=\parskip,topsep=3pt}
\setdescription{itemsep=0pt,partopsep=0pt,parsep=\parskip,topsep=3pt}
\setlist[itemize]{leftmargin=35pt}
\setlist[enumerate]{leftmargin=35pt}
\usepackage{changepage} 
\newcommand{\checks}[1]{{\color{black}{#1}}} 

\usepackage{contour} 
\usepackage{xcolor}
\contourlength{0.03em}
\contournumber{auto}

\newtheorem{theoremAlph}{Theorem}

\newtheorem{theorem}{Theorem}[section]
\newtheorem{lemma}[theorem]{Lemma}
\newtheorem{corollary}[theorem]{Corollary}
\newtheorem{main theorem}[theorem]{Main Theorem}
\newtheorem{proposition}[theorem]{Proposition}
\newtheorem{definition}[theorem]{Definition}

\newtheorem{remark}[theorem]{Remark}
\newtheorem{example}[theorem]{Example}

\numberwithin{equation}{section}

\def\orcid{
\begin{tikzpicture}[baseline=-1mm]
\filldraw[Green!35] (0,0) circle (5pt);
\filldraw[white] (0,0) node{\tiny\textbf{iD}};
\end{tikzpicture}
}

\def\Hom{\mathrm{Hom}}

\newcommand{\defines}[1]{{\it\color{blue!75}{#1}}}
\title{\bf Kleisli convolution representations of power monoids \footnote{The authors are normally listed alphabetically by surname.
In this submission, however, Haicun Wen is placed as the first author to satisfy the doctoral graduation requirement of Northwest Normal University.}
}

\vspace{5mm}

\author{
Haicun Wen$^{\ref{Author1}, \href{https://orcid.org/0009-0007-6552-1703}{\orcid}\ref{orcid1}}$, ~
Jian He$^{\ref{Author2}, \href{https://orcid.org/0009-0007-6001-1851}{\orcid}\ref{orcid2}}$, ~
Yu-Zhe Liu$^{\ref{Author3}, \href{https://orcid.org/0009-0005-1110-386X}{\orcid}\ref{orcid3}~\ref{CorrespondingAuthor}}$
}
\date{ }
\begin{document}






\maketitle

\begin{enumerate}[label=\textbf{\color{red}\arabic*}] \footnotesize
  \item
    \begin{center}
      College of Mathematics and Statistics, Northwest Normal University, Lanzhou 730070, China;

      E-mail: \url{Wenhaicun1990@163.com} (H. Wen)
    \end{center} \label{Author1}
 \item
    \begin{center}
      Department of Applied Mathematics, Lanzhou University of Technology, Lanzhou 730050, China;

      E-mail:  \url{jianhe30@163.com} (J. He)
    \end{center} \label{Author2}
  \item
    \begin{center}
      School of Mathematics and Statistics, Guizhou University,
      Guiyang 550025, Guizhou, China;

      E-mail:  \url{liuyz@gzu.edu.cn} / \url{yzliu3@163.com} (Y.-Z Liu)
    \end{center} \label{Author3}
\end{enumerate}
\vspace{2mm}
\begin{enumerate}[label=\textbf{\color{red}$\dag$}]
  \item \footnotesize
    \begin{center}
      Corresponding author
    \end{center} \label{CorrespondingAuthor}
\end{enumerate}
\vspace{2mm}
\begin{enumerate} \footnotesize
    \item[{\orcid}] \centering  ORCID: \href{https://orcid.org/0009-0007-6552-1703}{0009-0007-6552-1703}
      \label{orcid1}
    \item[{\orcid}] \centering  ORCID: \href{https://orcid.org/0009-0007-6001-1851}{0009-0007-6001-1851}
      \label{orcid2}
    \item[{\orcid}] \centering  ORCID: \href{https://orcid.org/0009-0005-1110-386X}{0009-0005-1110-386X}
      \label{orcid3}
\end{enumerate}

\vspace{1mm}

\begin{adjustwidth}{1cm}{1cm}
  \noindent \footnotesize
  \textbf{Abstract}:
We show that power semigroups of groups, and more generally reduced finitary power monoids, arise naturally as convolution monoids in Kleisli categories of powerset monads:
\begin{itemize}
  \item for the non-empty powerset monad, the Kleisli Hom-space $\mathrm{Hom}_{\mathbf{Kl}(\mathscr P_+)}(1,G)$ is isomorphic to the power monoid $\mathcal P_+(G)$;
  \item for the reduced finite powerset monad on pointed sets, the Kleisli Hom-space $\mathrm{Hom}_{\mathbf{Kl}(\mathscr P_{\mathrm{fin}})}$ $(\mathbb Z/2\mathbb Z,H)$ is isomorphic to the reduced finitary power monoid $\mathcal P_{\mathrm{fin},1}(H)$.
\end{itemize}
This unifies several constructions in power semigroup theory: Kleisli convolution representations of semigroups, base change along surjective group homomorphisms, and rigidity of automorphism groups.
As an application, we prove that for every proper numerical monoid $S$, the Kleisli Hom-monoid $\mathrm{Hom}_{\mathbf{Kl}(\mathscr P_{\mathrm{fin}})}(\mathbb Z/2\mathbb Z,S)$ is rigid. This gives a proof of the Tringali--Yan conjecture in the language of Kleisli categories. It should be mentioned that this conjecture was already proved by Bhowmik and Tringali in a preprint posted on arXiv on July 25, 2026.

\vspace{1mm}

\label{2020MSC}

\vspace{1mm}

     \label{Keywords}
\end{adjustwidth}

\newpage
\setcounter{tocdepth}{2}
\tableofcontents


\def\compos{{\begin{smallmatrix}\circ\end{smallmatrix}}}

\section{Introduction}

\def\Set{\mathbf{Set}}
\def\Cat{\mathbf{C}}
\def\Obj{\mathrm{Obj}}
\def\Mor{\mathrm{Mor}}
\def\compos{\ \lower-0.2ex\hbox{\tikz\draw (0pt, 0pt) circle (.1em);} \ }
\def\Setstar{\mathbf{PSet}}
\def\Kl{\mathbf{Kl}}
\def\Relplus{\mathbf{Rel}_+}
\def\Pplus{\mathcal P_+}
\def\fin{\mathrm{fin}}
\def\Pfin{\mathcal P_{\mathrm{fin}}}
\def\Pfinred{\mathcal P_{\mathrm{fin},1}}
\def\Aut{\operatorname{Aut}}
\def\id{\operatorname{id}}
\def\Sub{\operatorname{Sub}}
\def\Ima{\operatorname{Im}}
\def\Ker{\operatorname{Ker}}
\newcommand{\To}[1]{\mathop{-\!\!\!-\!\!\!\longrightarrow}\limits^{#1}}
\newcommand{\oT}[1]{\mathop{\longleftarrow\!\!\!-\!\!\!-}\limits^{#1}}

Let $H$ be a semigroup.  The multiplication of $H$ extends to its non-empty subsets by $AB:=\{ab:a\in A,\ b\in B\}$.
The resulting power semigroup is a classical construction in semigroup theory, going back at least to the works of Tamura and Shafer in \cite{TS1967}.
One of its basic structural questions \emph{asks to what extent the power semigroup determines the original semigroup}.
This question led to the study of globally determined classes and related isomorphism problems, see for example
\cite{Mog1973, GI1984, GanZhao2014, Tri2025, Rago2026Counterexample, Rago2026}.
Standard background on semigroups and their representations may be found in \cite{How1995}.
For a monoid $H$, $\Pfinred(H)=\{A\subseteq H:A\text{ is finite and }1_H\in A\}$ is also a monoid,
and we call it the reduced finitary power monoid of $H$.
Power monoids have recently been studied from the viewpoints of factorization theory and arithmetic combinatorics
\cite{FT2018, BG2025, CT2026, DGHLS2025, Tri2026}.
For numerical monoids, the ambient additive order supplies particularly effective tools, where the required background on numerical monoids is given in \cite{ADGS2020}.
Automorphism groups and rigidity phenomena for finitary or reduced power monoids have been determined in several important cases,
including $\mathbb N$, $\mathbb Z$, finite abelian groups, and numerical monoids \cite{Rago2025, TW2025, TY2025, TW2026, WXZZ2025}.
Related results concerning torsion groups and the Bienvenu--Geroldinger conjecture appear in \cite{TriYan2025, TY2026}.

The Kleisli construction is one of the standard categorical constructions associated with a monad, and it is the origin of Kleisli categories, which were initially introduced for the study of adjoint functors, see \cite[etc]{K1965, EM1965}.
The main purpose of this paper is to show that power semigroups of groups, and more generally reduced finitary power monoids, are naturally instances of convolution in Kleisli categories.
To be precise, for the set category $\Set$, one can define two special triples
$\mathscr P_+ = (\Pplus, \eta, \mu)$ and $\mathscr P_{\mathrm{fin}} = (\Pfin, \eta, \mu)$
and obtain two Kleisli categories $\Kl(\mathscr P_+)$ and $\Kl(\mathscr P_{\mathrm{fin}})$, respectively.
Then the two Kleisli categories have natural monoidal category structures $(\Set, \otimes, \{*\})$,
and, for any two Kleisli morphisms $f, g : X \rightsquigarrow Y$ (say also Kleisli arrows) in $\Hom_{\Kl(\mathscr P_+)}(X, Y)$
(the case for $f,g\in \Hom_{\Kl(\mathscr P_{\mathrm{fin}})}(X, Y)$ is similar),
we have a convolution
\[ f\star g := m \compos_{\mathscr T} (f\otimes g) \compos_{\mathscr T} \Delta: \quad
  \checks{
  \xymatrix{X \ar@{~>}[r]^{\Delta} & X\otimes X \ar@{~>}[r]^{f\otimes g}
  & Y\otimes Y \ar@{~>}[r]^{m} & Y,
  } } \]
which is a Kleisli arrow $X\rightsquigarrow Y$ lying in $\Hom_{\Kl(\mathscr P_+)}(X, Y)$.
Here, $\Delta$ is a comultiplication defined on $X$, and $m$ is a multiplication defines on $\Pplus(Y) :=\{A\subseteq Y: A\ne\varnothing\}$.
Based on the above, we obtain the first main result of this paper.

\begin{theoremAlph}[{Theorem \ref{thm:convolution}, Lemma \ref{lemm:P+}}] \label{thm:A 260822}
For any group $G$, $\Pplus(G)$ and the Kleisli Hom-space $\Hom_{\Kl(\mathscr P_+)}(1,G)$ are monoids $($the multiplication in this Kleisli Hom-space is the convolution ``$\star$'' as above$)$, and we have an isomorphism
\[ \Hom_{\Kl(\mathscr P_+)}(1,G) \cong \Pplus(G) \]
of monoids. 
The same assertion holds with $\mathscr P_+$ replaced by $\mathscr P_{\fin}$.
\end{theoremAlph}

Here, the object $1$ in the Kleisli category $\Kl(\mathscr P_+)$ (or $\Kl(\mathscr P_{\fin})$) is a one-point set in $\Set$.
The categorical perspective given in Theorem \ref{thm:A 260822} makes the multiplication of subsets appear not as an ad hoc extension of the original multiplication, but as the canonical convolution arising from the monoidal structure of the Kleisli category.
Now, we have a natural question that what is $\Pplus(q)$ (or $\Pfin(q)$) for homomorphism $q:F\to G$ of groups $F$ and $G$?
The following theorem provides an answer for this question, and shows that: each Kleisli arrow in $\Hom_{\Kl(\mathscr P_+)}(1, G)$ can be lifted to a Kleisli arrow in $\Hom_{\Kl(\mathscr P_+)}(1, F)$; and this lift is convolution-preserving.

\begin{theoremAlph}[{Theorem \ref{thm:base-change}}] \label{thm:B 260822}
Every surjective homomorphism $q:F\to G$ of groups induces an embedding
\[ q^{\dag}: \Hom_{\Kl(\mathscr P_+)}(1, G) \hookrightarrow \Hom_{\Kl(\mathscr P_+)}(1, F) \]
sends each Kleisli point $f_A: 1=\{*\} \rightsquigarrow G$ $($defined by $f_A(*)=A\subseteq G)$ in $\Hom_{\Kl(\mathscr P_+)}(1, G)$ to the Kleisli point $f_{q^{-1}(A)}: 1 \rightsquigarrow F$ $($satisfying $f_{q^{-1}(A)}(*)={q^{-1}(A)}\subseteq F)$ in $\Hom_{\Kl(\mathscr P_+)}(1, F)$.
\end{theoremAlph}

Let $\Setstar$ be the \defines{pointed set category}, which can be seen as a subcategory of $\Set$ (up to some embedding),
whose object is a pairs $(X,*)$, say pointed set, of a set $X$ and an element $*\in X$,
and whose morphism $(X,*) \to (Y,\circ)$ is a map $f: X\to Y$ with $f(*)=\circ$.
Then we can define $\mathcal P_{\fin,*}(X,*):= \bigl(\{A\subseteq X:A\text{ is finite and }*\in A\},\{*\}\bigr)$ for each pointed set $(X,*)$.
We can naturally induce a triple $\mathscr P_{\fin}=(\mathcal P_{\fin,-}, \eta, \mu)$ over $\Setstar$ by using the monad $\mathscr P_{+,\fin}$ given in the previous text. In this paper, we show the following important fact.

\begin{theoremAlph}[{Theorem \ref{thm:Pfin-monad}}]
The triple $\mathscr P_{\fin}=(\mathcal P_{\fin,-}, \eta, \mu)$ is a monad over $\Setstar$.
\end{theoremAlph}

This theorem shows that we can build a Kleisli category $\Kl(\mathscr P_{\fin})$ over $\Setstar$.
Another purpose of this paper is to discuss numerical monoids and their reduced finitary power monoids under this Kleisli category $\Kl(\mathscr P_{\fin})$, and we can study the rigidity of the automorphism group of power monoids.
Here, for a monoid $H = (H,\cdot)$, its \defines{reduced finitary power monoid} is defined as
$\Pfinred(H):=\{A\subseteq H:A\text{ is finite and }1\in A\}$.
The following theorem provides a description of the reduced finitary power monoid of a monoid $H$.

\begin{theoremAlph}[{Theorem \ref{thm:Kleisli=Pfin}}]
For every monoid $H$, pointed by its identity element $1_H$, there is a natural bijection
\[ \theta_{\fin,H}: \Hom_{\Kl(\mathscr P_{\fin,*})}(\mathbb{Z}/2\mathbb{Z},H) \to \mathcal P_{\fin,1_H}(H)\]
of monoids.
\end{theoremAlph}

Parallel to Theorem \ref{thm:B 260822}, we have the following theorem.

\begin{theoremAlph}[{Theorem \ref{thm:rigidity-transfer}}]
Let $q:F \twoheadrightarrow \mathbb Z$ be a surjective group homomorphism. Then, for any numerical monoid $S$, the following statements hold.
\begin{enumerate}[label={\rm(\arabic*)}]
  \item The map
    \begin{center}
      $\rho_q: \mathcal P_{\fin,0}(S) \cong \Hom_{\Kl(\mathscr P_{\fin,*})}(\mathbb{Z}/2\mathbb{Z},S) \to \Pplus(F) \cong \Hom_{\Kl(\mathscr P_{\fin,*})}(\mathbb{Z}/2\mathbb{Z},F)$, $\rho_q(A):=q^{-1}(A)$
    \end{center}
    is an injective semigroup homomorphism.
  \item Conjugation by $\rho_q$ induces an isomorphism $\Aut(M_S) \to \Aut(\rho_q(M_S))$ of groups.
  \item For all $A,B,X\in M_S$, $A\star X=B\star X$ if and only if $\rho_q(A)\star\rho_q(X)=\rho_q(B)\star\rho_q(X)$.
\end{enumerate}
\end{theoremAlph}


Finally, we consider an application to the Tringali--Yan conjecture.
Tringali and Yan conjectured that $\mathcal P_{\fin,0}(S)$ is rigid whenever $S$ is a numerical monoid properly contained in $\mathbb N$, see \cite[Section 4]{TY2025}.
This conjecture was proved by Bhowmik and Tringali in \cite[Theorem 1.2]{BhowmikTringali}.
Their proof is shorter than the Kleisli proof given in this paper. By using the isomorphism
\[ \Hom_{\Kl(\mathscr P_{\fin})}(\mathbb Z/2\mathbb Z,S) \cong \mathcal P_{\fin,0}(S), \]
we restate the relevant arguments in terms of Kleisli arrows and Kleisli convolution.
In particular, we prove that
\[ \Aut( \Hom_{\Kl(\mathscr P_{\fin})}(\mathbb Z/2\mathbb Z,S))\cong\{1\} \]
for every numerical monoid $S$ properly contained in $\mathbb N$, see Theorem \ref{thm:direct-Kleisli-rigidity}.
The purpose of this application is to show that \emph{Kleisli convolution representations may provide a useful categorical viewpoint for problems in semigroup theory}.


\section{Powerset monads and convolution} \label{sect:convol}

\textsl{Throughout the paper, a semigroup need not have an identity, whereas a monoid
does.  For a set $X$, let $\Pplus(X):=\{A\subseteq X:A\ne\varnothing\}$
and let $\Pfin(X):=\{A\subseteq X:0<|A|<\infty\}$.
The symbols $\Pplus$ and $\Pfin$ will retain these meanings throughout.
In particular, $\Pfin(X)$ never includes the empty set.}

\subsection{Powerset monads and their Kleisli categories}

\begin{definition}\label{def:Klcat}\rm
A \defines{monad} over a category $\Cat$ is a triple $\mathscr T= (T,\eta,\mu)$ defined as follows.
\begin{enumerate}[label=(M\arabic*)]
  \item $T$ is a functor $T: \Cat \to \Cat$,
    $\mu$ is a natural transformation $\mu: T \compos T \to T$,
    and $\eta$ is a natural transformation $\eta: \id_{\Cat} \to T$ such that
    \[\mu_X \compos T(\eta_X) = \id_{TX},
      \quad \mu_X\compos\eta_{TX} = \id_{TX}
      \quad \text{and}\quad \mu_X\compos T\mu_X = \mu_X\compos \mu_{TX}\]
    hold for all object $X \in \Obj(\Cat)$.
    \label{M1}
  \item The elements in $\mathscr T$ are morphisms of the form $a: X \to TY$ which are called \defines{Kleisli arrows}.
    \label{M2}
  \item For any two elements $f: X_1\to TX_2$ and $g: X_2 \to TX_3$,
    the multiplication is defined by the \defines{Kleisli composition}
    \[g \compos_{\mathscr T} f := \Big(\xymatrix{
     X_1 \ar[r]^{f} & TX_2 \ar[r]^{Tg} & T\compos T ~ X_3 \ar[r]^{\mu_{X_3}} & T X_3
    }\Big). \]
    \label{M3}
  \item The identity morphism $\id_X$ in $\mathscr T$ is $\eta_X$.
    \label{M4}
\end{enumerate}
Furthermore, given a monad $\mathscr T$, we define the \defines{Kleisli category} of $\Cat$, denoted by $\Kl(\mathscr T)$,
is a category whose objects and morphisms are given by
\begin{itemize}
  \item $\Obj(\Cat_T) := \Obj(\Cat)$,
  \item $\Hom_{\Cat_T}(X,Y) := \Hom_{\Cat}(X,TY)$,
\end{itemize}
and the composition of two morphisms is defined by Kleisli composition.
Each morphism in Kleisli category is called a \defines{Kleisli morphism}.
\end{definition}

Let $\Set$ be a set category whose objects are sets and morphisms are maps.
Let
\[ \mathscr P_+
= (\Pplus,\, \eta:\id_{\Set}\to\Pplus, \, \mu:\Pplus\checks{\times}\Pplus\to \Pplus) \]
be the triple defined by:
\begin{enumerate}[label=($\mathscr{P}$\arabic*)]
  \item for any set $X$ in $\Set$ and any $x\in X$, $\eta_X(x)=\{x\}$ is a set containing only one element $x$;
    \label{P1}
  \item for any $\mathcal A \in \Pplus(\Pplus(X))$, $\mu_X(\mathcal A)$ is the union $\bigcup\limits_{A\in\mathcal A}A$ $(\in \Pplus(X))$.
    \label{P2}
\end{enumerate}
Here, for a map $f:X\to Y$, we define $\Pplus(f)$ is a correspondence
\[ \Pplus(f): \Pplus(X) \to \Pplus(Y) \]
sending each $A$ to the \defines{direct image} $f[A]:=\{f(a)\in Y : a\in A\} \in \Pplus(Y)$.

\begin{lemma}\label{lemm:non-empty powerset monad}
The triple $\mathscr P_+=(\Pplus,\eta,\mu)$ is a monad over $\Set$.
\end{lemma}

\begin{proof}
First of all, for any $X\in\Obj(\Set)$, we have $\Pplus(X)$ is a set.
For the identity map $\id_X:X\to X$, and for any non-empty $A\subseteq X$,
we have $\Pplus(\id_X)(A)=\id_X[A]=A$, so $\Pplus(\id_X)=\id_{\Pplus(X)}$.
For composable maps $f:X\to Y$ and $g:Y\to Z$, and any $A\in\Pplus(X)$,
\[
\Pplus(g\compos f)(A)=(g\compos f)[A]=g[f[A]]=\Pplus(g)\bigl(\Pplus(f)(A)\bigr)
=(\Pplus(g)\compos\Pplus(f))(A).
\]
Thus $\Pplus(g\compos f)=\Pplus(g)\compos\Pplus(f)$. Hence $\Pplus$ is an endofunctor on $\Set$.

Next, we show that $\eta:\id_{\Set}\Rightarrow\Pplus$ and $\mu:\Pplus\circ\Pplus\Rightarrow\Pplus$ are natural transformations.
For any map $f:X\to Y$, we need $\Pplus(f)\compos\eta_X=\eta_Y\compos f$.
For any $x\in X$, $\Pplus(f)(\eta_X(x))=\Pplus(f)(\{x\})=\{f(x)\}=\eta_Y(f(x))$.
Thus the naturality of $\eta$ holds.
One can check that $\mu$ is a natural transformation in a similar way.

Third, we show that $\mu_X\compos\eta_{\Pplus(X)}=\id_{\Pplus(X)}$ and $\mu_X\compos\Pplus(\eta_X)=\id_{\Pplus(X)}$.
Indeed, for any $A\in\Pplus(X)$, we have $\mu_X(\eta_{\Pplus(X)}(A))=\mu_X(\{A\})=\bigcup\limits_{B\in\{A\}}B=A$, so
\begin{align}\label{eq:non-empty powerset monad 1}
  \mu_X\compos\eta_{\Pplus(X)}=\id_{\Pplus(X)}.
\end{align}
Also, $\Pplus(\eta_X)(A)=\eta_X[A]=\{\{x\}:x\in A\}$,
hence we obtain $\mu_X(\Pplus(\eta_X)(A))=\bigcup\limits_{x\in A}\{x\}=A$. Therefore,
\begin{align}\label{eq:non-empty powerset monad 2}
  \mu_X\compos\Pplus(\eta_X)=\id_{\Pplus(X)}
\end{align}
Let $\mathfrak A\in\Pplus(\Pplus(\Pplus(X)))$, i.e. $\mathfrak A$ is a non-empty collection of non-empty collections of non-empty subsets of $X$.
Then $  \Pplus(\mu_X)(\mathfrak A)
= \{\mu_X(\mathcal A):\mathcal A\in\mathfrak A\}
= \Big\{\bigcup\limits_{A\in\mathcal A}A:\mathcal A\in\mathfrak A \Big\}.$
Thus
\[ \mu_X\bigl(\Pplus(\mu_X)(\mathfrak A)\bigr)
 = \bigcup_{\mathcal A\in\mathfrak A}
   \bigg(\bigcup_{A\in\mathcal A}A\bigg). \]
On the other hand, we have $\mu_{\Pplus(X)}(\mathfrak A)=\bigcup\limits_{\mathcal A\in\mathfrak A}\mathcal A$,
then
\[ \mu_X\bigl(\mu_{\Pplus(X)}(\mathfrak A)\bigr)
 = \mu_X\bigg(\bigcup_{\mathcal A\in\mathfrak A}\mathcal A\bigg)
 = \bigcup_{B\in\bigcup\limits_{\mathcal A\in\mathfrak A}\mathcal A}B. \]
Note that the elements $B$ of $\bigcup\limits_{\mathcal A\in\mathfrak A}\mathcal A$ are exactly the sets $A$ that belong to some $\mathcal A\in\mathfrak A$. Hence
\[
\bigcup_{B\in\bigcup\limits_{\mathcal A\in\mathfrak A}\mathcal A}B
=\bigcup_{\mathcal A\in\mathfrak A}\bigcup_{A\in\mathcal A}A.
\]
Therefore the two expressions coincide, proving
\begin{align}\label{eq:non-empty powerset monad 3}
  \mu_X\compos\Pplus(\mu_X)=\mu_X\compos\mu_{\Pplus(X)}.
\end{align}
By \eqref{eq:non-empty powerset monad 1}, \eqref{eq:non-empty powerset monad 2}
and \eqref{eq:non-empty powerset monad 3}, we obtain \ref{M1}.

We have verified only \ref{M1}, because \ref{M2}--\ref{M4} are not axioms requiring proof.
Clause \ref{M2} merely introduces the terminology ``Kleisli arrows'' for morphisms $X\to TY$;
clause \ref{M3} defines the Kleisli composition rule using the already-established $\mu$;
and clause \ref{M4} defines the identity morphism using the already-established $\eta$. These are constructions that become available once the monad structure from \ref{M1} has been proven.
They are not additional conditions that $\mathscr P_+$ must satisfy.
\end{proof}

\begin{definition} \rm
The triple $\mathscr P_+=(\Pplus,\eta,\mu)$ given in Lemma \ref{lemm:non-empty powerset monad} is called a \defines{$($non-empty$)$ powerset monad} on $\Set$.
\end{definition}

The unit laws and the associativity law of the monad are respectively the identities
\[ \bigcup_{x\in A}\{x\}=A,
 \quad
 \bigcup\{A\}=A,
 \quad
 \bigcup_{\mathcal A\in\mathfrak A}\bigcup_{A\in\mathcal A}A = \bigcup_{A\in\bigcup\mathfrak A}A, \]
where $A\in\Pplus(X)$, $\mathcal A \in \Pplus(\Pplus(X))$, and $\mathfrak A \in \Pplus(\Pplus(\Pplus(X)))$ are arbitrary.
The Kleisli category $\Kl(\mathscr P_+)$ is a category whose objects are sets,
and a Kleisli morphism $f$ from $X$ to $Y$ can be seen as a relation
$f:X\rightsquigarrow Y$ assigns a non-empty subset $f(x)$ of $Y$ to each $x\in X$, i.e.,
\[ X \To{f} \Pplus(Y), \quad x\mapsto f(x)\subseteq Y \]
Then the composite of $f:X\rightsquigarrow Y$ and $g:Y \rightsquigarrow Z$ is
$(g\compos_{\mathscr P_+} f)(x)=\bigcup\limits_{y\in f(x)}g(y)$, i.e.,
\[ X \To{f} \Pplus(Y)
     \To{\Pplus(g)} \Pplus(\Pplus(Z))
     \To{\mu_Z} \Pplus(Z), \]
where for $x\in X$,
\begin{align*}
   \mu_X\compos \Pplus(g) \compos f(x)
& = \mu_X(g[f(x)])
  = \mu_X(\{g(y): y\in f(x)\}) \\
& \mathop{=\!=}\limits^{\text{\ref{P1}}}
    \bigcup_{A\in \{g(y): y\in f(x)\} }A
  = \bigcup\limits_{y\in f(x)}g(y).
\end{align*}

\begin{remark} \rm
Recall that a monad $(T,\eta,\mu)$ is called a \defines{strong monad} if it together with a natural transformation
\[ \operatorname{st}_{X,Y}: X \times TY \longrightarrow T(X \times Y), \]
called the \defines{strength}, satisfying the following \defines{strength axioms}:
\begin{enumerate}[label=(SM\arabic*)]
  \item $\operatorname{st}_{X,Y} \circ (\id_X \times \eta_Y) = \eta_{X \times Y}$,
    \label{SM1}
  \item $\operatorname{st}_{X,Y} \circ (\id_X \times \mu_Y)
       = \mu_{X \times Y} \circ T(\operatorname{st}_{X,Y}) \circ \operatorname{st}_{X, TY}$,
    \label{SM2}
  \item and the associativity condition with the product associator, which in $\Set$ is automatic.
    \label{SM3}
\end{enumerate}
The monad $\mathscr P_+$ is strong for the Cartesian monoidal structure on $\Set$.
Indeed, for any sets $X,Y$, define
\[
\operatorname{st}_{X,Y}: X \times \Pplus(Y) \to \Pplus(X \times Y),
\quad (x,A)\mapsto \{(x,a):a\in A\} ~(\subseteq X \times Y).
\]
We have the following two facts:
\begin{itemize}
  \item For \ref{SM1}, take any $(x,y)\in X\times Y$. Then
\begin{align}\label{eq:fact260819 1}
  \operatorname{st}_{X,Y}(x,\eta_Y(y)) = \operatorname{st}_{X,Y}(x,\{y\}) = \{(x,y)\} = \eta_{X\times Y}(x,y)
\end{align}
  \item For \ref{SM2}, take any $x\in X$ and $\mathcal A\in \Pplus(\Pplus(Y))$. Then we have
\begin{align*}
  \operatorname{st}_{X,Y}(x,\mu_Y(\mathcal A))
= \operatorname{st}_{X,Y}\bigg(x,\bigcup_{A\in\mathcal A}A\bigg)
= \bigg\{(x,y): y\in \bigcup_{A\in\mathcal A}A\bigg\}.
\end{align*}
and
\begin{align*}
  & \big(\mu_{X \times Y} \compos
    \Pplus(\operatorname{st}_{X,Y}) \compos
    \operatorname{st}_{X, \Pplus(Y)}\big)
    (x, \mathcal A) \\
=~& \big(\mu_{X \times Y} \compos
    \Pplus(\operatorname{st}_{X,Y})\big)
    (\{(x,A): A\in\mathcal A\}) \\
=~& \bigcup_{A\in\mathcal A}\{(x,a):a\in A\}
 =  \bigg\{(x,y): y\in \bigcup_{A\in\mathcal A}A \bigg\}.
\end{align*}
Thus both sides of \ref{SM2} are equal.
  \item The condition \ref{SM3} follows similarly and is immediate from the associativity of the Cartesian product of sets.
\end{itemize}
For the purposes of this paper, we do not need the axioms \ref{SM1}--\ref{SM3}.
\end{remark}

\begin{remark} \rm
The powerset monad $\mathscr P_+$ is commutative, i.e., for any $A\in\Pplus(X)$ and $B\in\Pplus(Y)$, the two possible orders of applying the strength followed by $\mu$ both yield the same subset $A\times B \in \Pplus(X\times Y)$.
Thus commutativity here is merely the fact that the Cartesian product of two non-empty subsets is independent of the order in which they are considered.
\end{remark}


\subsection{Convolutions of Kleisli morphisms}

\begin{definition} \rm
Recall that a \defines{monoidal category} consists of a category $\Cat$, a bifunctor
$\otimes:\Cat\times\Cat\to\Cat$ called the \defines{tensor product},
an object $I$ called the \defines{unit},
and three natural isomorphisms
\[
\alpha_{X,Y,Z}:(X\otimes Y)\otimes Z \xrightarrow{\sim} X\otimes(Y\otimes Z),
\quad
\lambda_X:I\otimes X\xrightarrow{\sim}X,
\quad
\rho_X:X\otimes I\xrightarrow{\sim}X,
\]
called the \defines{associator}, \defines{left unitor}, and \defines{right unitor} respectively.
These are required to satisfy the pentagon and triangle identities,
which ensure that the tensor product is associative and unital in a coherent manner.
\end{definition}

\begin{example}\rm\label{examp:set}\rm
The category $\mathbf{Set}$ carries the \defines{Cartesian monoidal structure},
where $\otimes=\times$ is the usual Cartesian product,
$I$ is a one-point set, and the isomorphisms are the evident identifications
\[((x,y),z)\mapsto(x,(y,z)),\quad (*,x)\mapsto x,\quad (x,*)\mapsto x.\]
In this case the coherence axioms \ref{SM1}, \ref{SM2} and \ref{SM3} are automatically satisfied.
\end{example}


\begin{definition}\rm
Let $(\Cat,\otimes,I)$ be a monoidal category.
A \defines{comonoid object} in $\Cat$ is a triple $(B,\Delta,\varepsilon)$ consisting of an object $B$, a morphism $\Delta:B\to B\otimes B$ which is called a \defines{comultiplication}, and a morphism $\varepsilon:B\to I$ which is called a \defines{counit}, such that the following diagrams commute:
\[
\begin{gathered}
\xymatrix@C=2cm{
  B \ar[r]^{\Delta} \ar[d]_{\Delta} & B\otimes B \ar[d]^{\Delta\otimes\id_B} \\
  B\otimes B \ar[r]_{\id_B\otimes\Delta} & B\otimes B\otimes B
}
\quad
\xymatrix{
  I\otimes B \ar@{<-}[r]^{\varepsilon \otimes\id_B} \ar@{<-}[dr]_{\lambda_B^{-1}} & B\otimes B \ar@{<-}[d]^{\Delta} & B \otimes I \ar@{<-}[l]_{\id_B\otimes \varepsilon } \ar@{<-}[dl]^{\rho_B^{-1}} \\
  & B &
}
\end{gathered}
\]
i.e., $\Delta$ is coassociative and $\varepsilon$ is a counit.
It is \defines{cocommutative} if $\Delta = \sigma_{B,B}\compos \Delta$, where $\sigma$ is the symmetry isomorphism of the monoidal category (when it exists).
\end{definition}

\begin{example}\rm\label{examp:monoid-in-set}
In the Cartesian monoidal category $\mathbf{Set}$, a monoid object is precisely a monoid in the usual algebraic sense.
Indeed, a monoid object $(H,m,e)$ consists of a set $H$, a map $m:H\times H\to H$ which is associative,
and a map $e:\{*\}\to H$ ($\{*\}$ is a one-point set) which selects the identity element $e(*)=1_H$ of $H$.
Thus, clearly, every group $G$, with its group multiplication $m:G\times G\to G$ and unit element $1_G$, is also a monoid object in $\mathbf{Set}$.
\end{example}

\begin{definition}\rm
Let $(\Cat,\otimes,I)$ be a monoidal category.
A \defines{monoid object} in $\Cat$ is a triple $(H,m,e)$ consisting of an object $H$, a morphism $m:H\otimes H\to H$ which is called a \defines{multiplication}, and a morphism $e:I\to H$ which is called a \defines{unit}, such that the following diagrams commute:
\[
\begin{gathered}
\xymatrix@C=1.65cm{
  \checks{H\otimes H\otimes H} \ar[r]^{\checks{\id_H\otimes m}} \ar[d]_{m\otimes\id_H} & H\otimes H \ar[d]^{\checks{m}} \\
  H\otimes H \ar[r]_{m} & H
}
\quad
\xymatrix{
  I\otimes H \ar[r]^{e\otimes\id_H} \ar[dr]_{\lambda_H} & H\otimes H \ar[d]^{m} & H\otimes I \ar[l]_{\id_H\otimes e} \ar[dl]^{\rho_H} \\
  & H &
}
\end{gathered}
\]
i.e., $m$ is associative and $e$ is a two-sided unit.
It is \defines{commutative} if $m = m\compos\sigma_{H,H}$.
\end{definition}

\begin{example}\rm\label{examp:comonoid-in-set}
In the Cartesian monoidal category $\mathbf{Set}$, every set $B$ carries a unique comonoid structure given by
$\Delta:B\to B\times B, x\mapsto (x,x)$ and
$\varepsilon:B\to \{*\}, x\mapsto *$.
The comultiplication $\Delta$ is coassociative because both paths yield $(x,(y,z))$ and $((x,y),z)$, which are identified under the associator.
It is cocommutative because $\Delta(x)=(x,x)$ is symmetric.
The counit $\varepsilon$ satisfies the counitality conditions since $(\varepsilon\times\id_B)(\Delta(x))=(*,x)$, which is identified with $x$ under $\lambda_B$, and similarly for the right counit.
In particular, the one-point set $\{*\}$ is a cocommutative comonoid with $\Delta:\{*\}\to \{*\}\otimes \{*\} := \{*\}\times \{*\}$, $\Delta(*)=(*,*)$, and $\varepsilon:\{*\} \to \{*\}$, $\varepsilon(*)=*$.
\end{example}

Next, we recall the definition of convolution.

\begin{definition} \label{def:convolution}\rm
Let $(\Cat,\otimes,I)$ be a monoidal category, let $(B,\Delta,\varepsilon)$ be a comonoid object, and let $(H,m,e)$ be a monoid object.
The \defines{convolution} on $\Hom_{\Cat}(B,H)$ is
\[ f\star g := m\compos(f\otimes g)\compos\Delta: \quad
  B \To{\Delta} B\otimes B \To{f\otimes g} H \otimes H \To{m} H. \]
\end{definition}

\begin{lemma}[{\cite[Chap 8, 8.3.1]{AguiarMahajan2020}}] \label{lemm:AM2020}
The convolution on $\Hom_{\Cat}(B,H)$ is associative and has a unit $e\compos\varepsilon$.
Furthermore, it is commutative when $B$ is cocommutative and $H$ is commutative.
\end{lemma}

Let $G$ be a group, written multiplicatively.
We use $1$ to represent the one-point set and use $1_G$ to represent the identity element of $G$.
The one-point set is a cocommutative comonoid under its diagonal, and $G$ is a monoid object under group multiplication.

\begin{definition}\rm
For the powerset monad $\mathscr P_+$, a \defines{Kleisli point} of a set $X$ is a Kleisli morphism $1 \rightsquigarrow X$ in $\Hom_{\Kl(\mathscr P_+)}(1, X)$.
\end{definition}

\begin{lemma}\label{lemm:260820}
There is a bijection $\Hom_{\Kl(\mathscr P_+)}(1, X) \to \Pplus(X)$ for each set $X$.
\end{lemma}

\begin{proof}
Each Kleisli point $f: 1=\{*\} \rightsquigarrow X$ in $\Hom_{\Kl(\mathscr P_+)}(1, X)$ provides a subset of $X$ by using its image $\Ima(f)$, naturally.
Then we obtain a correspondence $\sigma: f \mapsto f(*) \in \Pplus(X)$, where $\Ima(f)=f(1)=\{f(*)\}$.
For each $A\in\Pplus(X)$, define $\tau: \Pplus(X) \to \Hom_{\Kl(\mathscr P_+)}(1, X)$ is the map sending $A$ to the Kleisli point $1 \rightsquigarrow X$, $*\mapsto A$.
Then we have
\[ (\tau\compos\sigma)(f: * \mapsto f(*)) = \tau(f(*)) = (1 \rightsquigarrow X, *\mapsto f(*)) = f \]
for each $f\in \Hom_{\Kl(\mathscr P_+)}(1, X)$ and
\[ (\sigma\compos\tau)(A) = \sigma(1 \rightsquigarrow X, *\mapsto A) = A \]
for each $A\in \Pplus(X)$,
i.e., we have $\tau\compos\sigma=\id_{\Hom_{\Kl(\mathscr P_+)}(1, X)}$ and $\sigma\compos\tau = \id_{\Pplus(X)}$ as required.
\end{proof}

The following result shows that Kleisli points of a group $G$ are in bijection with non-empty subsets of $G$, i.e., $\Hom_{\Kl(\mathscr P_+)}(1,G) \cong \Pplus(G)$.

\begin{theorem}\label{thm:convolution}
There is a bijection
\[ \theta_{+,G}: \Hom_{\Kl(\mathscr P_+)}(1,G) \to \Pplus(G), \quad (f: 1=\{*\} \to \Pplus(G)) \mapsto f(*)\]
Under this identification, convolution is given by
\begin{equation}\label{eq:convolution}
 A\star B=\{ab:a\in A,\ b\in B\}.
\end{equation}
The same assertion holds with $\Pplus$ replaced by $\Pfin$.
\end{theorem}

\begin{proof}
By Lemma \ref{lemm:260820}, a Kleisli point $1\rightsquigarrow G$ in $\Hom_{\Kl(\mathscr P_+)}(1,G)$ is determined by its value at the unique element of $1$, and this value is an arbitrary non-empty subset of $G$.
Then, naturally, each subset $A$ of $G$ can be seen as a Kleisli point $f_A$ in $\Hom_{\Kl(\mathscr P_+)}(1,G)$.
Therefore, we can compute the convolution
\[ A\star B: \quad
  1 \To{\Delta} 1\otimes 1 \mathop{=}^{\spadesuit} 1\times 1
  \To{f_A\otimes f_B} G \otimes G \mathop{=}^{\clubsuit} G \times G
  \To{m} G \]
for any two subsets $A$ and $B$ of $G$, where $\spadesuit$ and $\clubsuit$ holds by Example \ref{examp:comonoid-in-set}.
Let $f_A, f_B: 1=\{*\} \to \Pplus(G)$ be the Kleisli morphisms corresponding $A$ and $B$ respectively,
we have then $f_A(*)=A$ and $f_B(*) = B$. The comultiplication $\Delta: 1 \to 1\times 1:=1\otimes 1$ sends $\{*\}$ to $(*,*)$, and applying $f_A\otimes f_B$ gives $A\times B$, to be precise,
\[ (f_A\otimes f_B)(*,*) = (f_A\otimes f_B)(*\otimes *)= f_A(*)\otimes f_B(*) = A\otimes B =A\times B.  \]
It follows that
\begin{align*}
    A \star B & = (m \compos (f_A\otimes f_B) \compos \Delta)(*) \\
& = m[A\times B] \\
& = m[\{(a,b):a\in A, b\in B\}] \\
& = \{ab:a\in A, b\in B\}.
\end{align*}
This proves \eqref{eq:convolution}.
Finiteness is preserved by Cartesian products and direct images, so the same proof applies to $\Pfin$.
\end{proof}

\begin{remark}\label{rem:quantale} \rm
If the empty set is added, the complete lattice $\mathcal P(G)$ becomes a unital quantale: convolution distributes over arbitrary unions.
The non-empty part used here is closed under convolution but is not a complete lattice.
The distinction matters for representation questions, because the papers considered here use non-empty subsets.
\end{remark}

\section{Kleisli convolution representations} \label{sect:Kle convol repr}

\textsl{In this section, we introduce Kleisli convolution representations of semigroups into power semigroups of groups, generalising the classical notion of representations by subsets. We then establish a ``base-change'' theorem (see Theorem \ref{thm:base-change}), showing that every surjective group homomorphism lifts a homomorphism of Kleisli convolution representations}

\subsection{Kleisli convolution representations of semigroups}

\begin{definition}\rm\label{def:Kleisli conv repr}
Let $S$ be a semigroup and $G$ a group.  A \defines{Kleisli convolution representation} of $S$ in $G$ is a homomorphism
\[ \rho:S \to \Hom_{\Kl(\mathscr P_{+,\mathrm{fin}})}(1, G), 
  \quad \rho(st)=\rho(s)\star\rho(t) \]
of semigroups. It is \defines{faithful} if it is injective
and \defines{disjoint} if $\rho(s)\cap\rho(t)=\varnothing$ whenever $s\ne t$.
\end{definition}

\begin{remark}\rm
By Theorem \ref{thm:convolution}, this is precisely a representation of $S$ in the convolution semigroup of the Kleisli points of the group object $G$.
Thus, an embedding $S\hookrightarrow \Hom_{\Kl(\mathscr P_{+,\mathrm{fin}})}(1, G)$ and a faithful Kleisli convolution representation are the same datum, not merely analogous data.
In particular, when we think that $\Pplus(G)$ is a monoid, then the isomorphism $\theta_{+,G}$ provide a representation
$\theta_{+,G}^{-1}: \Pplus(G) \To{\cong} \Hom_{\Kl(\mathscr P_+)}(1,G)$
of $\Pplus(G)$, obviously.
\end{remark}

\begin{remark} \rm
For $\rho(s)\subseteq G$, define a \defines{relation $R$} on $S\times G$ as follows:
\begin{align}\label{relation}
  s~R~g \text{~if and only if~} g\in\rho(s).
\end{align}
This gives a purely relational description of disjoint power-semigroup embeddings.
The Kleisli convolution representation $\rho:S \to \Hom_{\Kl(\mathscr P_{+,\mathrm{fin}})}(1, G)$ of a semigroup $S$ can be encoded equivalently as the relation \eqref{relation}.
In this relational language, the homomorphism condition and the disjointness condition become, respectively:
\begin{enumerate}[label=(\Alph*)]
  \item \defines{Multiplication law}: can be written without reference to $\rho$ as
    $(st)~R~g$  if and only if $\exists x,y\in G$ such that $s R x,\ t R y,\text{~and~} g=xy.$
    \label{260820 1}
  \item \defines{Disjointness}: the converse relation $R^{\mathrm{op}}\subseteq G\times S$ is single-valued wherever it is defined.
    \label{260820 2}
\end{enumerate}
See Table \ref{tab:mult-disj}.
\begin{table}[htbp]
\centering
\begin{tabular}{ccc}
\hline
\textbf{Kleisli convolution representation} && \textbf{Relational language} \\ \hline
$\rho:S\to \Hom_{\Kl(\mathscr P_{+,\mathrm{fin}})}(1, G)$ && $R\subseteq S\times G$, $s\,R\,g \iff g\in\rho(s)$ \\
$\rho(st)=\rho(s)\star\rho(t)$ && $st\,R\,g \iff \exists x,y: s\,R\,x,\ t\,R\,y,\ g=xy$ \\
$\rho(s)\cap\rho(t)=\varnothing$ $(s\ne t)$ && $R^{\mathrm{op}}$ is single-valued \\ \hline
\end{tabular}
\caption{Correspondence between Kleisli convolution representation and relational language}
\label{tab:mult-disj}
\end{table}
Conversely, given a relation $R$ on $S\times G$ satisfying \ref{260820 1} and \ref{260820 2}
such that for every $s\in S$, there is a $g\in G$ with $sRg$.
One recovers the representation by $\rho(s)=\{g\in G:s\,R\,g\}$.
Thus, \emph{a disjoint Kleisli convolution representation is the same datum as a relation whose multiplication law holds and whose converse is single-valued}.
The following is a theorem proved by Bershadsky and Kublanovsky \cite[Theorem 1.1]{BK2025}, expressed in the terminology above:
\begin{itemize}
  \item {\rm(\cite[Theorem 1.1, (1)]{BK2025})}
    every semigroup has a disjoint faithful Kleisli convolution representation in a symmetric group of sufficiently large degree;
  \item {\rm(\cite[Theorem 1.1, (2)]{BK2025})}
    every semigroup also has a disjoint faithful Kleisli convolution representation in a free group with sufficiently many free generators.
\end{itemize}
\end{remark}

\subsection{Embeddings of Hom-spaces in Kleisli category}

Let $F$ and $G$ be two groups and $q:F\to G$ be a homomorphism of groups. For $A\subseteq G$, write
\[ q^{-1}(A):= \{ x\in F : q(x)\in A \}.\]

\begin{lemma}[{\cite{BK2025}}] \label{lemm:P+}
For each group $G$, we have $\Pplus(G)$ is a monoid.
\end{lemma}

\begin{proof}
By Theorem \ref{thm:convolution}, we have a bijection $\Hom_{\Kl(\mathscr P_+)}(1,G) \to \Pplus(G)$,
and $\Pplus(G)$ is a semigroup whose elements can be seen as Kleisli points in $\Hom_{\Kl(\mathscr P_+)}(1,G)$
and whose multiplication is induced by the convolution in $\Hom_{\Kl(\mathscr P_+)}(1,G)$.
Moreover, $\Pplus(G)$ has an identity. Thus, $\Pplus(G)$ is a monoid.
\end{proof}

\begin{proposition} \label{prop:base-change}
If a homomorphism $q:F\to G$ of groups is surjective, then the following statements hold.
\begin{enumerate}[label={\rm(\arabic*)}]
  \item The map $q^{-1}: \Pplus(G)\to\Pplus(F)$ is an injective homomorphism of semigroups, i.e.,
    $q^{-1}(A\star B)=q^{-1}(A)\star q^{-1}(B)$ for all non-empty subsets $A,B\subseteq G$;
  \item Disjoint families remain disjoint after applying $q^{-1}$.
\end{enumerate}
\end{proposition}

\begin{proof}
First of all, we have that $\Pplus(G)$ and $\Pplus(F)$ are monoids by Lemma \ref{lemm:P+}.
By the definition of $q^{-1}$, we have that $q^{-1}$ sends each $A \in \Pplus(G)$ to
the subset $q^{-1}(A) = \{ x \in F : q(x)\in A\}$ of $F$,
then $\Ima(q^{-1}) \subseteq \Pplus(F)$, i.e., $q^{-1}$, as a map from the monoid $\Pplus(G)$ to the monoid $\Pplus(F)$, is well-defined.

Take two arbitrary subset $A$ and $B$ in $\Pplus(G)$.
For any $x\in q^{-1}(A)$ and $y\in q^{-1}(B)$, we have $q(xy)=q(x)q(y)\in A\star B$.
It follows $q^{-1}(A)\star q^{-1}(B)\subseteq q^{-1}(A\star B)$.
Conversely, take $z\in q^{-1}(A\star B)$, we can find two element $a\in A$ and $b\in B$ such that $q(z)=ab$.
Since $q$ is surjective, there exists $x\in F$ with $q(x)=a$, then $q(x)^{-1}q(z)=q(x^{-1}z)=a^{-1}ab=b$,
and so, $x^{-1}z\in q^{-1}(B)$. Note that $x\in q^{-1}(A)$, then $z=x\cdot x^{-1}z \in q^{-1}(A)\star q^{-1}(B)$.
It follows $q^{-1}(A)\star q^{-1}(B)\supseteq q^{-1}(A\star B)$.
Therefore, we have $q^{-1}(A)\star q^{-1}(B) = q^{-1}(A\star B)$.

If $q^{-1}(A)=q^{-1}(B)$ and $a\in A$, choose $x\in F$ with $q(x)=a$. Then $x\in q^{-1}(B)$, so $a=q(x)\in B$.
This proves $A\subseteq B$, and the opposite inclusion is symmetric. Hence $q^{-1}$ is injective.
Finally, inverse images preserve intersections, so they preserve disjointness.
\end{proof}

Now we illustrate Proposition \ref{prop:base-change} with a concrete example (using additive notation, since the groups involved are additive).
This example shows that $q^{-1}$ may be not a homomorphism of monoids.

\begin{example}\rm\label{ex:base-change}\rm
Let $F=\mathbb Z$, $G=\mathbb Z/2\mathbb Z = \{0,1\}$, and $q:\mathbb Z\to \mathbb Z/2\mathbb Z$ be the canonical quotient map $n\mapsto n\bmod 2$.
Then $q$ is surjective, and the inverse-image map
$q^{-1}:\Pplus(\mathbb Z/2\mathbb Z)\longrightarrow \Pplus(\mathbb Z)$ is given by
\begin{align*}
& q^{-1}(\{0\}) = \{ n\in\mathbb Z : n\text{ is even} \} = 2\mathbb Z,\\
& q^{-1}(\{1\}) = \{ n\in\mathbb Z : n\text{ is odd} \} = 2\mathbb Z+1,\\
& q^{-1}(\{0,1\}) = \mathbb Z.
\end{align*}
Thus $q^{-1}$ is an injective homomorphism of semigroups, as guaranteed by Proposition \ref{prop:base-change}.
For instance, we have $q^{-1}(\{0\})\star q^{-1}(\{0\}) = 2\mathbb Z + 2\mathbb Z = 2\mathbb Z = q^{-1}(\{0\}\star\{0\})$
because $\{0\}\star\{0\}=\{0\}$ in $\Pplus(\mathbb Z/2\mathbb Z)$.
Similarly, for $\{0\}$ and $\{1\}$, we have $q^{-1}(\{0\})\star q^{-1}(\{1\}) = 2\mathbb Z + (2\mathbb Z+1) = 2\mathbb Z+1 = q^{-1}(\{0\}\star\{1\})$, where $\{0\}\star\{1\}=\{1\}$.
In this example, $q^{-1}$ does not preserve the monoid identity.
The identity of $\Pplus(\mathbb Z/2\mathbb Z)$ is $\{0\}$, while the identity of $\Pplus(\mathbb Z)$ is $\{0\}$.
But $q^{-1}(\{0\}) = 2\mathbb Z \neq \{0\}$.
Therefore, $q^{-1}$ is a homomorphism of semigroups but not of monoids.
\end{example}

\begin{theorem}\label{thm:base-change}
Every surjective homomorphism $q:F\to G$ of groups induces an embedding
\[ q^{\dag}: \Hom_{\Kl(\mathscr P_+)}(1, G) \hookrightarrow \Hom_{\Kl(\mathscr P_+)}(1, F) \]
sends each Kleisli point $f_A: 1=\{*\} \rightsquigarrow G$ $($defined by $f_A(*)=A\subseteq G)$ in $\Hom_{\Kl(\mathscr P_+)}(1, G)$ to the Kleisli point $f_{q^{-1}(A)}: 1 \rightsquigarrow F$ $($satisfying $f_{q^{-1}(A)}(*)={q^{-1}(A)}\subseteq F)$ in $\Hom_{\Kl(\mathscr P_+)}(1, F)$.
\end{theorem}

\begin{proof}
By Proposition \ref{prop:base-change}, we have an injective homomorphism $q^{-1}: \Pplus(G)\to\Pplus(F)$ of semigroups.
Theorem \ref{thm:convolution} admits two isomorphisms
\begin{center}
  $\sigma_G:\Hom_{\Kl(\mathscr P_+)}(1, G) \To{\cong} \Pplus(G)$
  and $\sigma_H:\Hom_{\Kl(\mathscr P_+)}(1, F) \To{\cong} \Pplus(F)$.
\end{center}
Then we obtain a composition of homomorphisms
\[ \Hom_{\Kl(\mathscr P_+)}(1, G) \To{\sigma_G} \Pplus(G) \To{q^{-1}} \Pplus(F)
   \To{\sigma_F^{-1}} \Hom_{\Kl(\mathscr P_+)}(1, F)\]
which is $q^{\dag}$ as required. Here, for each $f_A: 1\rightsquigarrow G$, $*\mapsto A$, we have
$\sigma_F\compos q^{-1}\compos \sigma_G(f_A)=\sigma_F(q^{-1}(A))=f_{q^{-1}(A)}$
where $f_{q^{-1}(A)}$ is the Kleisli point lying in $\Hom_{\Kl(\mathscr P_+)}(1,F)$
whose image contains only one element $q^{-1}(A)$.
Finally, since Proposition \ref{prop:base-change} yields that $q^{-1}$ is injective, so is $q^{\dag}$.
\end{proof}

There is a corollary worth showing here.

\begin{corollary}\label{coro:lift-representation}
If $S$ has a faithful {\rm(}resp., disjoint faithful{\rm)} Kleisli convolution representation in $G$, then every surjective group homomorphism $q:F\twoheadrightarrow G$ lifts it to a faithful {\rm(}resp., disjoint faithful{\rm)} representation in $F$.
\end{corollary}

\begin{proof}
Let $\rho:S\to\Pplus(G)$ be a faithful (respectively, disjoint faithful) Kleisli convolution representation of $S$ in $G$.
Define $\rho' := q^{-1}\compos \rho: S \to \Pplus(F)$, which send each $s\in S$ to $\rho'(s)=q^{-1}(\rho(s))\subseteq F$.
Since $\rho(s)\neq\varnothing$ for each $s\in S$ and $q$ is surjective, each $\rho'(s)$ is non-empty.
Thus $\rho'$ is well-defined.

For any $s,t\in S$, we have
\[ \rho'(st)=q^{-1}(\rho(st))
\mathop{=}\limits^{\spadesuit} q^{-1}(\rho(s)\star\rho(t))
\mathop{=}\limits^{\clubsuit} q^{-1}(\rho(s))\star q^{-1}(\rho(t))
  = \rho'(s)\star\rho'(t),\]
where $\spadesuit$ uses that $\rho$ is a Kleisli convolution representation of $S$ (see Definition \ref{def:Kleisli conv repr}), and $\clubsuit$ is given by Proposition \ref{prop:base-change}.
Hence $\rho'$ is a Kleisli convolution representation of $S$ in $F$.

If $\rho$ is faithful (i.e., injective), then $\rho'=q^{-1}\compos \rho$ is also injective,
since $q^{-1}$ is injective by Proposition \ref{prop:base-change}. Thus, $\rho'$ is faithful.
If $\rho$ is disjoint faithful, then $\rho$ is injective and the family $\{\rho(s):s\in S\}$ is pairwise disjoint.
By Proposition \ref{prop:base-change}, inverse images preserve disjointness,
then the family $\{\rho'(s):s\in S\}$ is also pairwise disjoint.
Thus, $\rho'$ is disjoint faithful.
\end{proof}

%
%
%

\section{Kleisli category over finite powerset monad} \label{sect:finite powerset monad}

\textsl{In this section, we construct the reduced finite powerset monad on the category $\Setstar$ of pointed sets, and prove that the reduced finitary power monoid $\mathcal P_{\fin,1}(H)$ of a monoid $H$ is isomorphic to the Kleisli Hom-monoid $\Hom_{\Kl(\mathscr P_{\fin})}(\mathbb Z/2\mathbb Z,H)$. This provide a Kleisli convolution representation of $\mathcal P_{\fin,1}(H)$, and identification allows automorphism problems for reduced power monoids to be studied in the language of Kleisli categories.}

\subsection{Finite powerset monad and their Kleisli categories}

Let $H=(H,\cdot)$ be a monoid. Recall that the \defines{reduced finitary power monoid} of $H$ is
\[ \Pfinred(H):=\{A\subseteq H:A\text{ is finite and }1\in A\}.\]
To express this construction without changing the notation for its elements, we work in the category $\Setstar$ of pointed sets.

For each set $X$, we can define a \defines{pointed set} $(X,*)$ for an element $*\in X$. Here, we call $*$ is the \defines{basepoint}.
Now, we consider the \defines{pointed set category} $\Setstar$ whose objects are pointed sets and whose morphisms $f:(X,*) \to (Y,\circ)$ are a map $f:X\to Y$ with the \defines{basepoint condition} $f(*)=\circ$.
Note that there is a natural correspondence $(X,*)\mapsto X$, which induces a faithful embedding $\Setstar \hookrightarrow \Set$.
Clearly, it is a forgetful functor. In view of this embedding, $\Setstar$ may be regarded as a subcategory of $\Set$. For simplicity, we shall always treat $\Setstar$ as a subcategory of $\Set$ in our paper.

For a pointed set $(X,*)$, we define
\[ \mathcal P_{\fin,*}(X,*):= \bigl(\{A\subseteq X:A\text{ is finite and }*\in A\},\{*\}\bigr). \]
and
\begin{enumerate}[label=($\mathscr{P}_{\text{fin}}$\arabic*)]
  \item for any set pointed $(X,*)$ in $\Setstar$ and any $x\in X$,
    $\eta_{(X,*)}(x) := \eta_X(x)\cup\{*\} = \{x,*\}$;
    \label{Pfin1}
  \item for any $\mathcal{A}\in \mathcal P_{\fin,*}(\mathcal P_{\fin,*}(X,*))$, $\mu_{(X,*)}(\mathcal A)$ is the union $\bigcup\limits_{A\in\mathcal A}A$.
    \label{Pfin2}
\end{enumerate}
Then one can check that
\[ \mathcal P_{\fin,-}: \Setstar \to \Setstar, \quad (X,*) \mapsto \mathcal P_{\fin,*}(X,*) \]
is a functor, because, for each pointed sets $(X,*)$, $(Y,\circ)$, $(Z,\bullet)$ and morphism $f: X\to Y$, $g: Y\to Z$,
the equation
\begin{align*}
 \mathcal P_{\fin,-}(g\compos f)(A)
& =(g\compos f)[A]=g[f[A]] \\
& =\mathcal P_{\fin,-}(g)(\mathcal P_{\fin,-}(f)(A)) \\
& =\bigl(\mathcal P_{\fin,-}(g)\compos \mathcal P_{\fin,-}(f)\bigr)(A)
   \quad (\forall A\in\mathcal P_{\fin,*}(X,*))
\end{align*}
admits that the following diagram
\[ \xymatrix{
& \mathcal P_{\fin,\circ}(Y,\circ)
  \ar[rd]^{\mathcal P_{\fin,-}(g)}
& \\
  \mathcal P_{\fin,*}(X,*)
  \ar[ru]^{\mathcal P_{\fin,-}(f)}
  \ar[rr]_{\mathcal P_{\fin,-}(g\compos f)}
&
& \mathcal P_{\fin,\bullet}(Z,\bullet)
} \]
commutes.

In this paper, $\mathcal P_{\fin,-}$ denotes the reduced finite powerset functor on $\Setstar$, sending a pointed set $(X,*)$ to
$\mathcal P_{\fin,*}(X,*) = \{A\subseteq X : A \text{ finite and } *\in A\}$.
Here the subscript ``$-$'' is a placeholder for the basepoint. When evaluated at a concrete object $(X,*)$, it is replaced by $*$, yielding $\mathcal P_{\fin,*}(X,*)$. Similarly, $\mathcal P_{\fin,-}(f)$ denotes the action of the functor on a morphism $f$. Thus $\mathcal P_{\fin,-}$ is the abstract functor, while $\mathcal P_{\fin,*}(X,*)$ is its value at a specific object.

We have the following result which shows that
the triple $\mathscr P_{\fin}=(\mathcal P_{\fin,-}, \eta, \mu)$ is a monad.

\begin{theorem}\label{thm:Pfin-monad}
The triple $\mathscr P_{\fin}=(\mathcal P_{\fin,-}, \eta, \mu)$ is a monad over $\Setstar$.
\end{theorem}

\begin{proof}
We have show that $\mathcal P_{\fin,-}$ is a endofunctor on $\Setstar$, with action on morphisms given by direct image
$\mathcal P_{\fin,*}(f)(A)=f[A]=\{f(a):a\in A\}$ for any $f:(X,*)\to(Y,\circ)$ and $A\in\mathcal P_{\fin,*}(X,*)$.
This is well-defined because finite subsets are mapped to finite subsets, and $f(*)=\circ$ ensures $\circ\in f[A]$.

Next, for any morphism $f:(X,*)\to(Y,\circ)$ in $\Setstar$ and any $x\in X$, we have
\[
  \mathcal P_{\fin,-}(f)(\eta_{(X,*)}(x))
= \mathcal P_{\fin,-}(f)(\{x,*\})
= \{f(x), f(*)\}
= \{f(x), \circ\}
= \eta_{(Y,\circ)}(f(x)).
\]
It follows that $\mathcal P_{\fin,-}(f)\compos \eta_{(X,*)} = \eta_{(Y,\circ)}\compos f$. Thus, the naturality of $\eta$ holds.

For any $f:(X,*)\to(Y,\circ)$ and any $\mathcal A\in \mathcal P_{\fin,-}(\mathcal P_{\fin,-}(X,*))$, we have
\[
\mathcal P_{\fin,-}(f)(\mu_{(X,*)}(\mathcal A))
= \mathcal P_{\fin,-}(f)\bigg(\bigcup_{A\in\mathcal A} A\bigg)
= f\bigg[\bigcup_{A\in\mathcal A} A\bigg]
= \bigcup_{A\in\mathcal A} f[A].
\]
On the other hand,
\[
  \mu_{(Y,\circ)}\Big(\mathcal P_{\fin,-}(\mathcal P_{\fin,-}(f))(\mathcal A) \Big)
= \mu_{(Y,\circ)}(\{\mathcal P_{\fin,-}(f)(A): A\in\mathcal A\})
= \mu_{(Y,\circ)}(\{f[A]: A\in\mathcal A\})
= \bigcup_{A\in\mathcal A} f[A].
\]
Then the two sides of $\mathcal P_{\fin,-}(f)\compos \mu_{(X,*)} = \mu_{(Y,\circ)}\compos \mathcal P_{\fin,-}(\mathcal P_{\fin,-}(f))$ coincide, proving the naturality of $\mu$.

Next, we show that $\mathscr P_{\fin}$ satisfies \ref{M1}.

For any $A\in\mathcal P_{\fin,*}(X,*)$, we have
$\mathcal P_{\fin,*}(\eta_{(X,*)})(A)
= \{\eta_{(X,*)}(x): x\in A\}
= \{\{x,*\}: x\in A\}$.
Then, by $*\in A$, we have
$\mu_{(X,*)}(\mathcal P_{\fin,*}(\eta_{(X,*)})(A))
= \bigcup_{x\in A} \{x,*\}
= A\cup\{*\} = A$.
By the arbitrariness of $A$, we obtain $\mu_{(X,*)}\compos \mathcal P_{\fin,*}(\eta_{(X,*)})=\mathrm{id}$.

Note that $*$ is the basepoint of $\mathcal P_{\fin,*}(X,*)$, then $\eta_{\mathcal P_{\fin,*}(X,*)}(A) = \{A, \{*\}\}$.
Therefore, or any $A\in\mathcal P_{\fin,*}(X,*)$, we have
$\mu_{(X,*)}(\eta_{\mathcal P_{\fin,*}(X,*)}(A)) = \mu_{(X,*)}(\{A,\{*\}\}) = A\cup\{*\} = A$.
By the arbitrariness of $A$, we obtain $\mu_{(X,*)}\compos \eta_{\mathcal P_{\fin,*}(X,*)} = \mathrm{id}$.

Take any $\mathfrak A \in \mathcal P_{\fin,*}(\mathcal P_{\fin,*}(\mathcal P_{\fin,*}(X,*)))$, we have
\[\mathcal P_{\fin,*}(\mu_{(X,*)}(\mathfrak A))
= \{\mu_{(X,*)}(\mathcal A): \mathcal A\in\mathfrak A\}
= \bigg\{\bigcup_{A\in\mathcal A} A : \mathcal A\in\mathfrak A\bigg\}.\]
Applying $\mu_{(X,*)}$ gives
\[ \mu_{(X,*)}(\mathcal P_{\fin,*}(\mu_{(X,*)}(\mathfrak A)))
 = \bigcup_{\mathcal A\in\mathfrak A}\bigcup_{A\in\mathcal A} A.\]
On the other hand, $\mu_{\mathcal P_{\fin,*}(X,*)}(\mathfrak A) = \bigcup\limits_{\mathcal A\in\mathfrak A} \mathcal A$,
and then
\[
  \mu_{(X,*)} (\mu_{\mathcal P_{\fin,*}(X,*)}(\mathfrak A))
= \mu_{(X,*)}\bigg(\bigcup_{\mathcal A\in\mathfrak A}\mathcal A\bigg)
= \bigcup_{B\in \bigcup\limits_{\mathcal A\in\mathfrak A}\mathcal A} B
\mathop{=}\limits^{\spadesuit} \bigcup_{\mathcal A\in\mathfrak A}\bigcup_{A\in\mathcal A} A,
\]
where $\spadesuit$ holds since the elements of $\bigcup\limits_{\mathcal A\in\mathfrak A}\mathcal A$ are precisely the sets $A$ belonging to some $\mathcal A\in\mathfrak A$. Hence $\mu_{(X,*)} \compos \mu_{\mathcal P_{\fin,*}(X,*)} = \mu_{(X,*)} \compos \mathcal P_{\fin,*}\mu_{(X,*)}$.

All monad axioms are satisfied (\ref{M2}--\ref{M4} are not axioms requiring proof).
Therefore $\mathscr P_{\fin}=(\mathcal P_{\fin,-},\eta,\mu)$ is a monad over $\Setstar$.
\end{proof}

\subsection{Realising reduced finitary power monoids as Kleisli Hom-spaces}

The following lemma shows that it is a monoid under the Kleisli category $\Kl(\mathscr P_{\fin})$ established in Theorem \ref{thm:Pfin-monad}.

\begin{lemma} \label{lemm:Pfin(X,*)-monoid}
Under the Kleisli category $\Kl(\mathscr P_{\fin})$ of the monad $\mathscr P_{\fin}$,
for every monoid $H$ as a set in $\Setstar$, pointed by its identity element $1_H$,
the set $\mathcal P_{\fin,1_H}(H)$ is a monoid
{\rm(}under the multiplication $A\star B$, where the identity element is $\{1_H\}${\rm)}.
\end{lemma}

\begin{proof}
Let $A,B\in\mathcal P_{\fin,1_H}(H)$. Since $A$ and $B$ are finite,
the set $A\star B$ is finite. Moreover, $1_H=1_H1_H\in A\star B$,
and hence $A\star B\in\mathcal P_{\fin,1_H}(H)$.

On the other hand, for $A,B,C\in\mathcal P_{\fin,1_H}(H)$, the associativity of the
multiplication of $H$ gives
$(A\star B)\star C =\{(ab)c:a\in A,\ b\in B,\ c\in C\} = \{a(bc):a\in A,\ b\in B,\ c\in C\} =A\star(B\star C)$.
Finally, $\{1_H\}\star A=A=A\star\{1_H\}$. Therefore, $\mathcal P_{\fin,1_H}(H)$ is a monoid with identity $\{1_H\}$.
\end{proof}

%

Let $\mathbb{Z}/2\mathbb{Z}=(\{0,1\},0)$ be a monoid given by addition ``$+$''.
In this case, the identity of it is $0$.
Of course, the group structure is mentioned here only to single out the natural basepoint $0$;
in the actual proof, we never use the fact that $\mathbb Z/2\mathbb Z$ is a group,
and merely regard it as a two-element set $\{0,1\}$ with basepoint $0$.

\begin{lemma}\label{lemm:Hom-monoid}
For every monoid $H$ in $\Setstar$, pointed by its identity element $1_H$,
the Kleisli Hom-space $\Hom_{\Kl(\mathscr P_{\fin,*})}(\mathbb Z/2\mathbb Z,H)$ is a monoid
{\rm(}under the operation $(u\star v)(i):=u(i)\star v(i)$ for $i\in \mathbb Z/2\mathbb Z${\rm)}.
\end{lemma}

\begin{proof}
By Theorem \ref{thm:Pfin-monad}, we have established the Kleisli category
$\Kl(\mathscr P_{\fin})$ on the monad $P_{\fin,*}$.
This means that our discussion for
$\Hom_{\Kl(\mathscr P_{\fin})}(\mathbb Z/2\mathbb Z,H)$ is valid.
For $u,v\in\Hom_{\Kl(\mathscr P_{\fin})} (\mathbb Z/2\mathbb Z$, $H)$ and $i\in\mathbb Z/2\mathbb Z$, define $(u\star v)(i):=u(i)\star v(i)$.
By Lemma \ref{lemm:Pfin(X,*)-monoid}, $u(i)\star v(i)$ belongs to $\mathcal P_{\fin,1_H}(H)$.
Moreover, since $u$ and $v$ are pointed maps, we have $u(0)=v(0)=\{1_H\}$, and hence $(u\star v)(0)=\{1_H\}\star\{1_H\}=\{1_H\}$.
Thus, $u\star v$ is a pointed map.
For any $u,v,w$ and $i\in\mathbb Z/2\mathbb Z$, we have
$ ((u\star v)\star w)(i) = (u(i)\star v(i))\star w(i) = u(i)\star(v(i)\star w(i)) = (u\star(v\star w))(i)$.
Therefore, the convolution is associative.
The identity element is the pointed map $\mathbf e$ defined by $\mathbf e(0)=\mathbf e(1)=\{1_H\}$,
since for any $A\in\mathcal P_{\fin,1_H}(H)$, $\{1_H\}\star A=A=A\star\{1_H\}$.
Thus, $\Hom_{\Kl(\mathscr P_{\fin})}(\mathbb Z/2\mathbb Z,H)$ is a monoid under convolution.
\end{proof}

\begin{theorem}\label{thm:Kleisli=Pfin}
For every monoid $H$, pointed by its identity element $1_H$, there is a natural bijection
\[ \theta_{\fin,H}: \Hom_{\Kl(\mathscr P_{\fin,*})}(\mathbb{Z}/2\mathbb{Z},H) \to \mathcal P_{\fin,1_H}(H)\]
of monoids. Under this identification, convolution is the usual product of reduced finite subsets.
\end{theorem}

\begin{proof}
First of all, we have established the Kleisli category $\Kl(\mathscr P_{\fin})$ in Theorem \ref{thm:Pfin-monad}.
Thus, we can use Lemmas \ref{lemm:Pfin(X,*)-monoid} and \ref{lemm:Hom-monoid} in this proof.
By Lemma \ref{lemm:Pfin(X,*)-monoid}, we have that $\mathcal P_{\fin,1_H}(H)$ is a monoid,
and by Lemma \ref{lemm:Hom-monoid}, we have that $\Hom_{\Kl(\mathscr P_{\fin,*})}(\mathbb{Z}/2\mathbb{Z},H)$ is a monoid.

We restrict Theorem \ref{thm:convolution} from $\mathbf{Set}$ to $\Setstar$ by requiring all subsets to contain the basepoint, and then further restrict to finite subsets. This yields a bijection $\theta$
from $\Hom_{\Kl(\mathscr P_{\fin,*})}(\mathbb Z/2\mathbb Z, H)$ to $\mathcal P_{\fin,1_H}(H)$.
To be more precise, by Definition \ref{def:Klcat}, a Kleisli morphism $u:\mathbb Z/2\mathbb Z\rightsquigarrow H$ is a pointed map
$u:\mathbb Z/2\mathbb Z \to \mathcal P_{\fin,*}(H,1_H)$ in $\Setstar$.
Since $\mathbb Z/2\mathbb Z=\{0,1\}$ has basepoint $0$, the pointed condition gives $u(0)=\{1_H\}$.
Thus $u$ is determined solely by $u(1)$, which can be any element of $\mathcal P_{\fin,*}(H,1_H)$.
Hence evaluation at $1$ gives the above bijection sending $u$ to a subset $u(1)$ of $H$.
Furthermore, by Theorem \ref{thm:convolution}, this bijection identifies the convolution of Kleisli points with the usual subset product, i.e., $A\star B=\{ab:a\in A,\ b\in B\}$. Then $\theta_{\fin,H}$ is an isomorphism of monoids
\end{proof}

\begin{remark}\rm
In the additive case, the distinguished point is written $0$ and $\Pfinred(H)$ is correspondingly written $\mathcal P_{\fin,0}(H)$.
This is only the conventional additive replacement of the identity symbol, and no new object is introduced.
\end{remark}

\section{Faithful convolution representations} \label{sect:faith convol repr}

\textsl{This section introduces faithful Kleisli convolution representations and a rigidity-transfer theorem for reduced finitary power monoids.}

\subsection{Faithful group models and rigidity transfer} \label{subsect:rigidity transfer}

\begin{definition}\label{def:rigidity}\rm \
\begin{enumerate}[label={\rm(\arabic*)}]
  \item Let $M$ be a semigroup. We say that $M$ is \defines{rigid} if every
    automorphism of $M$ is the identity automorphism, or equivalently, $\Aut(M)=\{\id_M\}$.
  \item More generally, a Kleisli Hom-space $\Hom_{\Kl(\mathscr T)}(B,H)$ is called \defines{rigid}
    if it is rigid as a semigroup under convolution.
    If the convolution has an identity, then it is equivalently rigid as a monoid.
\end{enumerate}
\end{definition}

\begin{remark}\label{rem:rigidity-isomorphism} \rm
Rigidity is invariant under semigroup isomorphisms. Indeed, if
$\rho:M\to N$ is an isomorphism, then conjugation by $\rho$ gives a
group isomorphism
\[ \Aut(M)\to\Aut(N),\quad \varphi\mapsto\rho\compos\varphi\compos\rho^{-1}. \]
Consequently, $M$ is rigid if and only if $N$ is rigid.
\end{remark}

Let $S$ be a \defines{numerical monoid}, that is, an additive submonoid of $\mathbb N = \{0,1,2,\ldots\}$ with $|\mathbb N_0\backslash S| < \infty$. For simplicity, we define $M_S:=\mathcal P_{\fin,0}(S)$.
Its multiplication is the sumset convolution $A\star B=\{a+b:a\in A,\ b\in B\}$.
By Theorem \ref{thm:Kleisli=Pfin}, we have an isomorphism
$\theta_{\fin,S}: \Hom_{\Kl(\mathscr P_{\fin})}(\mathbb Z/2\mathbb Z,S) \cong M_S$
of monoids.

The inclusion $S\subseteq\mathbb Z$ gives an injective homomorphism
\[  M_S \mathop{\hookrightarrow}\limits^{\subseteq} \mathcal P_{\fin,0}(\mathbb Z),
   \quad A\mapsto A \]
of monoids. Thus, the reduced power monoid already has a canonical faithful convolution
representation in a group.

\begin{theorem}\label{thm:rigidity-transfer}
Let $q:F \twoheadrightarrow \mathbb Z$ be a surjective group homomorphism. Then the following statements hold.
\begin{enumerate}[label={\rm(\arabic*)}]
  \item The map $\rho_q:M_S \to\Pplus(F)$, $\rho_q(A):=q^{-1}(A)$ is an injective semigroup homomorphism.
    \label{rigidity-transfer 1}
  \item Conjugation by $\rho_q$ induces a isomorphism $\Aut(M_S) \to \Aut(\rho_q(M_S))$,
    $\varphi\mapsto\rho_q\compos\varphi\compos\rho_q^{-1}$ of groups.
    \label{rigidity-transfer 2}
  \item For all $A,B,X\in M_S$, $A\star X=B\star X$ if and only if $\rho_q(A)\star\rho_q(X)=\rho_q(B)\star\rho_q(X)$.
    \label{rigidity-transfer 3}
\end{enumerate}
\end{theorem}

\begin{proof}
Proposition \ref{prop:base-change}, applied to $q$, gives $q^{-1}(A\star B)=q^{-1}(A)\star q^{-1}(B)$.
It also gives injectivity of $\rho_q$, i.e., if $q^{-1}(A)=q^{-1}(B)$, then applying $q$ and using surjectivity yields $A=B$.
Thus, \ref{rigidity-transfer 1} holds.
Hence $\rho_q$ is an isomorphism from $M_S$ onto its image $\Ima(\rho_q) = \rho_q(M_S)$,
and conjugation by any isomorphism induces the isomorphism of automorphism groups by Remark \ref{rem:rigidity-isomorphism}.
Then we obtain \ref{rigidity-transfer 2}.
Finally, for all $A,B,X\in M_S$, $\rho_q(A)\star\rho_q(X)=\rho_q(B)\star\rho_q(X)$ is preserved because $\rho_q$ is a homomorphism and reflected because it is injective.
\end{proof}

\begin{remark}\rm
Notice that $\rho_q(\{0\})=\Ker(q)$, which need not be the identity of $\Pplus(F)$ but is the identity inside the image $\rho_q(M_S)$.
This is why Theorem \ref{thm:rigidity-transfer} \ref{rigidity-transfer 1} is naturally stated for semigroups, whereas the image $\Ima(\rho_q) = \rho_q(M_S)$ is nevertheless a monoid in its own right.
\end{remark}

The isomorphism $\Aut(M_S) \cong \Aut(\rho_q(M_S))$ given by Theorem~\ref{thm:rigidity-transfer} \ref{rigidity-transfer 2} implies that the rigidity of $M_S$ can be reduced to the rigidity of $\rho_q(M_S)$. The latter is isomorphic to a submonoid of $\mathcal P_+(F)$, which, by Theorem \ref{thm:convolution}, is equivalent to study some submonoid of Kleisli Hom-space $\Hom_{\Kl(\mathscr P_+)}(1,F)$.
Furthermore, Theorem \ref{thm:rigidity-transfer} shows that rigidity may be checked in any of the faithful group models above.
What remains is to produce equations in the monoid that determine each finite subset.
The next section does this without introducing an additional system of coordinates.

%

\subsection{Application: Tringali--Yan Conjucture}

Tringali and Yan conjectured in \cite[Section~4]{TY2025} that the reduced finitary power monoid $\mathcal P_{\fin,0}(S)$ is rigid for every numerical monoid $S$ properly contained in $\mathbb N$.
This conjecture was proved by Bhowmik and Tringali in \cite[Theorem 1.2]{BhowmikTringali}.
Their proof is shorter than the proof given in this subsection.
Some arguments below are Kleisli reformulations of the corresponding arguments in \cite[Section 3]{BhowmikTringali}.
More precisely, Lemma \ref{lem:two-point-Kleisli-fixed} is the Kleisli version of \cite[Lemma 3.1]{BhowmikTringali},
the second part of Lemma \ref{lem:Kleisli-cardinality-maximum} is the Kleisli version of \cite[Lemma 3.2]{BhowmikTringali},
and Lemma \ref{lem:Kleisli-fixed-tails} is the Kleisli version of \cite[Lemma 3.4]{BhowmikTringali}.
Our purpose here is not to claim another solution of the conjecture,
but to present the relevant arguments in the language of Kleisli categories and to illustrate the possible applications of Kleisli convolution representations in semigroup theory.

Let $S$ be a numerical monoid and put
$M_S=\mathcal P_{\fin,0}(S)$ as in
Section~\ref{sect:faith convol repr}. For every $A\in M_S$, we define
$f_A:=\theta_{\fin,S}^{-1}(A)$ in this section, where
\[
 \theta_{\fin,S}\colon
 \Hom_{\Kl(\mathscr P_{\fin})}(\mathbb Z/2\mathbb Z,S)
 \xrightarrow{\cong}M_S
\]
is the isomorphism of monoids given in
Theorem~\ref{thm:Kleisli=Pfin}. Thus, $f_A$ is the unique Kleisli
arrow
\[
 f_A\colon\mathbb Z/2\mathbb Z\rightsquigarrow S
\]
whose underlying map is given by
\[
 f_A(0)=\{0\}
 \quad\text{and}\quad
 f_A(1)=A.
\]

Let $S$ be a numerical monoid and put $M_S=\mathcal P_{\fin,0}(S)$ as in Section \ref{sect:faith convol repr}.
For every $A\in M_S$, we define $f_A:=\theta_{\fin,S}^{-1}(A)$ in this section,
where $\theta_{\fin,S}$ is the isomorphism
$\theta_{\fin,S}: \Hom_{\Kl(\mathscr P_{\fin})}(\mathbb Z/2\mathbb Z,S)\xrightarrow{\cong} M_S$
of monoids given in Theorem \ref{thm:Kleisli=Pfin}.
Obviously, $f_A$ is the unique Kleisli arrow $f_A:\mathbb Z/2\mathbb Z\rightsquigarrow S$ whose underlying map
$f_A: \mathbb Z/2\mathbb Z \to M_S = \mathcal P_{\fin,0}(S)$ is given by $f_A(0)=\{0\}$ and $f_A(1)=A$.

The following lemma is the Kleisli version of \cite[Lemma 3.1]{BhowmikTringali}.

\begin{lemma}\label{lem:two-point-Kleisli-fixed}
Let $S$ be a numerical monoid and let $\alpha\in\Aut( \Hom_{\Kl(\mathscr P_{\fin})}(\mathbb Z/2\mathbb Z,S))$.
For every $s\in S$, let $f_{\{0,s\}}$ be the Kleisli arrow $f_{\{0,s\}} : \mathbb Z/2\mathbb Z \rightsquigarrow S$ determined by $f_{\{0,s\}}(0)=\{0\}$, $f_{\{0,s\}}(1)=\{0,s\}$.
Then $\alpha(f_{\{0,s\}})=f_{\{0,s\}}$ for all $s\in S$.
In other words, every automorphism of $\Hom_{\Kl(\mathscr P_{\fin})}(\mathbb Z/2\mathbb Z,S)$ fixes all Kleisli arrows $\mathbb Z/2\mathbb Z \rightsquigarrow S$.
\end{lemma}

\begin{proof}
By Theorem \ref{thm:Kleisli=Pfin}, we have $\theta_{\fin,S}: \Hom_{\Kl(\mathscr P_{\fin})}(\mathbb Z/2\mathbb Z,S) \xrightarrow{\cong} M_S$,
and $\theta_{\fin,S}$ induces a conjugate isomorphism $\tilde{\theta}_{\fin,S}: \Aut(\Hom_{\Kl(\mathscr P_{\fin})}(\mathbb Z/2\mathbb Z,S)) \xrightarrow{\cong} \Aut(M_S)$, i.e., $\widetilde{\theta}_{\fin,S}(\alpha) := \theta_{\fin,S} \compos \alpha \compos \theta_{\fin,S}^{-1}$.
Thus, by \cite[Theorem 3.2 and Corollary 3.3]{TY2026}, there is a unique bijection $g:S\to S$ such that
$\widetilde{\theta}_{\fin,S}(\{0,s\})=\{0,g(s)\}$ for all $s\in S$. Note that $\theta_{\fin,S}(f_{\{0,s\}})=\{0,s\}$, then
\begin{align*}
   \theta_{\fin,S}(\alpha(f_{\{0,s\}}))
 = \widetilde{\theta}_{\fin,S}\bigl(\theta_{\fin,S}(f_{\{0,s\}})\bigr)
 = \widetilde{\theta}_{\fin,S}(\{0,s\})
 = \{0,g(s)\}
 = \theta_S(f_{\{0,g(s)\}}), \quad \forall s\in S.
\end{align*}
It follows that
\begin{equation}\label{eq:Kleisli-pullback}
  \alpha(f_{\{0,s\}}) = f_{\{0,g(s)\}} \text{ for all } s\in S.
\end{equation}
In particular, we have $g(0)=0$.
Notice that $S$ is a numerical monoid, then it is cancellative, and by \cite[Proposition 4.3]{TY2026},
we have $g(ns)=ng(s)$ for all $s\in S$ and $n\in\mathbb N$.
Let $0\ne a,b\in S$. The ordinary integer product $ab$ can be regarded both as the $b$-fold sum of $a$ and as the $a$-fold sum of $b$. Thus, we obtain $bg(a)=g(ba)=g(ab)=ag(b)$, and so $\frac{g(a)}a=\frac{g(b)}b$.
Hence there exists $\lambda\in\mathbb Q_{>0}$ such that$g(s)=\lambda s$ for all $s\in S$.
Assume $\min(S\setminus\{0\})=m$. Since multiplication by the positive rational number $\lambda$ is strictly increasing and $g$ is a bijection from $S$ to itself, $g(m)$ must again be the least positive element of $S$. Thus $g(m)=m$.
It follows that $\lambda=1$, and consequently, we obtain that
$g(s)= s$ for all $s\in S$. Finally, by \eqref{eq:Kleisli-pullback},
we obtain $\alpha(f_{\{0,s\}})=f_{\{0,s\}}$ for all $s\in S$ as required.
\end{proof}

The assertion concerning the maximum in the following lemma is the Kleisli version of \cite[Lemma 3.2]{BhowmikTringali}.

\begin{lemma}\label{lem:Kleisli-cardinality-maximum}
Let $S$ be a numerical monoid and let $\alpha\in\Aut( \Hom_{\Kl(\mathscr P_{\fin})}(\mathbb Z/2\mathbb Z,S))$.
For $A,A'\in M_S$ with $\alpha(f_A)=f_{A'}$, we have $|A'|=|A|$ and $\max A'=\max A$.
\end{lemma}

\begin{proof}
By Theorem~\ref{thm:Kleisli=Pfin}, we have an isomorphism $\theta_{\fin,S}: \Hom_{\Kl(\mathscr P_{\fin})} (\mathbb Z/2\mathbb Z,S) \xrightarrow{\cong} M_S$ of monoids and an isomorphism $\tilde{\theta}_{\fin,S}: \Aut(\Hom_{\Kl(\mathscr P_{\fin})} (\mathbb Z/2\mathbb Z,S)) \xrightarrow{\cong} \Aut(M_S)$, where $\tilde{\theta}_{\fin,S}=\theta_{\fin,S}\compos\alpha\compos\theta_{\fin,S}^{-1}$.
Then $\tilde{\theta}_{\fin,S}\in\Aut(M_S)$. Note that $\theta_{\fin,S}(f_A)=A$ and $\theta_{\fin,S}(f_{A'})=A'$,
then, by $\alpha(f_A)=f_{A'}$, we have
\[ \tilde{\theta}_{\fin,S}(A)
 = \theta_{\fin,S}(\alpha(\theta_S^{-1}(A)))
 = \theta_{\fin,S}(\alpha(f_A))
 = \theta_{\fin,S}(f_{A'})
 = A'.\]
Since $S$ is commutative and cancellative, and it contains an element of infinite order, \cite[Lemma~4]{Rago2026} implies that $\tilde{\theta}_{\fin,S}$ preserves cardinality of set. Then $|A'| = |\tilde{\theta}_{\fin,S}(A)| = |A|$.

Now, we need show $\max A'=\max A$.
If $A=\{0\}$, i.e., $A$ has only one element which is identity element of $M_S$, then $\tilde{\theta}_{\fin,S}(A)=A$, and so $A'=A$ in this case, we are done.
Thus, we may therefore assume that $m:=\max A>0$.

By \cite[Lemma 2.2]{TriYan2025}, for every sufficiently large integer $k$, we have
\begin{equation}\label{eq:stabilization-in-MS}
  \underbrace{A\star A \star \cdots \star A \star A}\limits_{\text{$k+1$ times}}
= \underbrace{A\star A \star \cdots \star A}\limits_{\text{$k$ times}}\star\{0,m\}.
\end{equation}
Lemma \ref{lem:two-point-Kleisli-fixed} admits $\alpha(f_{\{0,m\}}) = f_{\{0,m\}}$.
Therefore,
\[ \tilde{\theta}_{\fin,S}(\{0,m\})
 = \theta_{\fin,S}(\alpha(\theta_{\fin,S}^{-1}(\{0,m\})))
 = \theta_{\fin,S}(\alpha(f_{\{0,m\}}))
 = \theta_{\fin,S}(f_{\{0,m\}})
 = \{0,m\}. \]
Applying $\tilde{\theta}_{\fin,S}$ to \eqref{eq:stabilization-in-MS}, and using $\tilde{\theta}_{\fin,S} (A)=A'$, we obtain
\[ \underbrace{A'\star A' \star \cdots \star A' \star A'}\limits_{\text{$k+1$ times}}
 = \underbrace{A'\star A' \star \cdots \star A'}\limits_{\text{$k$ times}}\star\{0,m\}.\]
Then $(k+1)\max A' = k\max A'+m$, and so $\max A'=m=\max A$.
\end{proof}

The following lemma is the Kleisli version of \cite[Lemma 3.4]{BhowmikTringali}.

\begin{lemma}\label{lem:Kleisli-fixed-tails}
Let $S$ be a numerical monoid properly contained in $\mathbb N$, and let $\alpha\in\Aut( \Hom_{\Kl(\mathscr P_{\fin})}(\mathbb Z/2\mathbb Z,S))$.
Let $c(S):=1+\max(\mathbb Z\setminus S)$. 
For $a\geqslant c(S)$ and $N\geqslant a$, define $I[a,N]:=\{0\}\cup\{a,a+1,\ldots,N\}$.
Then, for every $a\geqslant c(S)$, there exists an integer $B_a\geqslant a$ such that
\[  N\geqslant  B_a \text{~yields that~} \alpha(f_{I[a,N]})=f_{I[a,N]}.\]
\end{lemma}

\begin{proof}
Note that $1\notin S$.
For every $a\geqslant c(S)$, we define $T_a:=\{0,a,a+1\}$, and assume $\alpha(f_{T_a})=f_Y$.
Then we have $\widetilde{\theta}_S(\alpha)(T_a)=Y$, equivalently.
By Lemma~\ref{lem:Kleisli-cardinality-maximum}, we have $|Y|=|T_a|=3$ and $\max Y=\max T_a=a+1$.
Thus, we have $T_a\star\{0,a\} = \{0,a,a+1,2a,2a+1\}$ by Theorem \ref{thm:convolution}, and have
\[ |Y\star\{0,a\}|
  \mathop{=}\limits^{\spadesuit} |\widetilde{\theta}_S(\alpha)(T_a)\star\widetilde{\theta}_S(\alpha)(\{0,a\})|
  \mathop{=}\limits^{{\color{red}\heartsuit}} |\widetilde{\theta}_S(\alpha)(T_a\star\{0,a\})|
  \mathop{=}\limits^{\clubsuit} |T_a\star\{0,a\}|
  \mathop{=}\limits^{{\color{red}\diamondsuit}} 5,\]
where $\spadesuit$ holds by the formula $\widetilde{\theta}_S(\alpha)(\{0,a\})=\{0,a\}$ given in Lemma \ref{lem:two-point-Kleisli-fixed},
$\color{red}\heartsuit$ holds since $\widetilde{\theta}_S(\alpha)$ preserves convolution,
$\clubsuit$ holds since $\widetilde{\theta}_S(\alpha)$ preserves cardinality,
and $\color{red}\diamondsuit$ is obtained by the fact that $a\geqslant c(S)\geqslant 2$ follows that the five elements in $T_a\star\{0,a\}$ are distinct.

On the other hand, $Y\star\{0,a\}=Y\cup(a+Y)$
and therefore $|Y\star\{0,a\}| = |Y|+|a+Y|-|Y\cap(a+Y)| = 2|Y|-|Y\cap(a+Y)|$.
Since $|Y|=3$ and $|Y\star\{0,a\}|=5$, we obtain $|Y\cap(a+Y)|=1$.
Then there exists $x\in Y$ such that $x+a\in Y$, and $ x+a\leqslant \max Y=a+1$.
Thus, $x\leqslant 1$. By $Y\subseteq S$ and $1\notin S$, we have $x=0$, and so $a\in Y$.

Moreover, $0\in Y$ because $Y\in M_S$, and $a+1\in Y$ because
$\max Y=a+1$. Since $|Y|=3$, we conclude that $Y=\{0,a,a+1\}=T_a$. Thus,
\begin{equation}\label{eq:Kleisli-fixed-tails 1}
  \widetilde{\theta}_S(\alpha)(T_a) = Y = T_a = \{0,a,a+1\}, \quad a\geqslant c(S).
\end{equation}
By \cite[Proposition 2.10]{TY2025}, we hvae
\begin{equation}\label{eq:Kleisli-fixed-tails 2}
  T_a\star T_{a+1}\star\cdots\star T_{2a} = I[a,N_a], \text{~where~} N_a=a(a+1)+\frac{(a+1)(a+2)}2
\end{equation}
By \eqref{eq:Kleisli-fixed-tails 1} and \eqref{eq:Kleisli-fixed-tails 2}, we obtain
\begin{align*}
   & \widetilde{\theta}_S(\alpha)(I[a,N_a]) \\
=\ & \widetilde{\theta}_S(\alpha)(T_a\star T_{a+1}\star\cdots\star T_{2a}) \\
=\ & \widetilde{\theta}_S(\alpha)(T_a)\star \widetilde{\theta}_S(\alpha)(T_{a+1}) \star \cdots \star \widetilde{\theta}_S(\alpha)(T_{2a}) \\
=\ & T_a\star T_{a+1}\star\cdots\star T_{2a} \\
=\ & I[a,N_a],
\end{align*}
i.e.,
\begin{equation}\label{eq:initial-tail-fixed}
  \widetilde{\theta}_S(\alpha)(I[a,N_a])=I[a,N_a].
\end{equation}
Notice that $N_a\geqslant 2a-1$. For every $N\geqslant 2a-1$, direct calculation gives
\begin{align}\label{eq:tail-increment-a}
 I[a,N]\star\{0,a\} = \{0\}\cup\{a,\ldots,N\}\cup\{2a,\ldots,N+a\} = I[a,N+a]
\end{align}
and
\begin{align} \label{eq:tail-increment-a+1}
   & I[a,N]\star T_a = I[a,N]\star\{0,a,a+1\} \nonumber \\
=\ & I[a,N+a]\cup\{a+1\}\cup\{2a+1,\ldots,N+a+1\} = I[a,N+a+1].
\end{align}
Hence \eqref{eq:initial-tail-fixed}--\eqref{eq:tail-increment-a+1} imply that
\[ \widetilde{\theta}_S(\alpha)(I[a,N_a+u])=I[a,N_a+u] \]
for every $u$ in the additive monoid $S'$ generated by $a$ and $a+1$.
Every integer $u\geqslant a(a-1)$ belongs to this additive monoid $S'$.
Indeed, assume $u=qa+r$, $0 \leqslant r < a$, then $u\geqslant a(a-1)$ implies $q\geqslant a-1\geqslant r$, and then $u=(q-r)a+r(a+1)$.
If $N\geqslant B_a := N_a+a(a-1)$, then $N-N_a\geqslant a(a-1)$, and consequently $\tilde{\theta}_{\fin,S} (\alpha)(I[a,N])=I[a,N]$.
Therefore, \[\theta_{\fin,S}(\alpha(f_{I[a,N]})) = \tilde{\theta}_{\fin,S}(\alpha)(\theta_{\fin,S}(f_{I[a,N]}))
 = \tilde{\theta}_{\fin,S}(\alpha)(I[a,N]) = I[a,N] = \theta_{\fin,S}(\alpha)(f_{I[a,N]}). \]
It follows that $\alpha(f_{I[a,N]})=f_{I[a,N]}$ since $\theta_{\fin,S}$ is injective.
\end{proof}

\begin{proposition}\label{prop:Kleisli-membership-preserved}
Let $S$ be a numerical monoid properly contained in $\mathbb N$, and let $\alpha\in\Aut($ $\Hom_{\Kl(\mathscr P_{\fin})}(\mathbb Z/2\mathbb Z,S))$.
If $A,A'\in M_S$ satisfies $\alpha(f_A)=f_{A'}$ then $y\in A$ if and only if $y\in A'$ {\rm(}$y\in S${\rm)}.
Consequently, we have $A=A'$ and $\alpha(f_A)=f_A$.
\end{proposition}

\begin{proof}
By $\alpha(f_A)=f_{A'}$, we have $\widetilde{\theta}_S(\alpha)(A) = (\theta_{\fin,S}\compos\alpha\compos\theta_{\fin,S}^{-1})(A)
= \theta_{\fin,S}(\alpha(f_A)) = \theta_{\fin,S}(f_A') = A'$ in this proof.
By Lemma \ref{lem:Kleisli-cardinality-maximum}, we have $\max A'=\max A$. We denote it by $\mathfrak{m}$ in this proof.

We first consider an element $y\in S$ satisfying $y\geqslant c(S)$.
Choose an integer $N\geqslant \max\{B_y,B_{y+1},y+\mathfrak{m}\}$.
We claim that, for every $X\in M_S$ with $\max X\leqslant \mathfrak{m}$, we have
\begin{equation}\label{eq:membership-by-translation}
  y\in X \quad\text{if and only if}\quad I[y,N]\star X=I[y+1,N]\star X.
\end{equation}
Indeed:
\begin{itemize}
  \item
For ``$\Rightarrow$'': Since $I[y+1,N]\subseteq I[y,N]$, the inclusion $I[y+1,N]\star X\subseteq I[y,N]\star X$ always holds.
Suppose that $y\in X$. Since $I[y,N]=I[y+1,N]\cup\{y\}$, it suffices to show that $y+X\subseteq I[y+1,N]\star X$.
If $x=0$, then we have $y=0+y\in I[y+1,N]\star X$ since $0\in I[y+1,N]$ and $y\in X$.
If $x>0$, then $y+1\leqslant y+x\leqslant y+\mathfrak{m} \leqslant N$. Hence $y+x\in I[y+1,N]$, and therefore $y+x=(y+x)+0\in I[y+1,N]\star X$. Thus, the equation $I[y,N]\star X=I[y+1,N]\star X$ holds.

  \item
Conversely, for ``$\Leftarrow$'':
If we have $I[y,N]\star X=I[y+1,N]\star X$, then we have $y=y+0\in I[y,N]\star X$, and then $y\in I[y+1,N]\star X$.
Thus, $y=h+x$ for some $h\in I[y+1,N]$ and $x\in X$.
If $h\ne0$, then $h\geqslant y+1$, it contradicts with $x\geqslant 0$.
Thus, $h=0$, and so we have $y=x\in X$. Therefore, \eqref{eq:membership-by-translation} holds.
\end{itemize}

By Lemma~\ref{lem:Kleisli-fixed-tails}, $\widetilde{\theta}_S(\alpha)(I[y,N])=I[y,N]$ and $\widetilde{\theta}_S(\alpha)(I[y+1,N])=I[y+1,N]$.
Since $\widetilde{\theta}_S(\alpha)$ is an injective homomorphism, the following equations are equivalent:
\begin{enumerate}[label={\rm(\arabic*)}]
  \item $I[y,N]\star A=I[y+1,N]\star A$;
    \label{26082113031}
  \item $\widetilde{\theta}_S(\alpha)(I[y,N]\star A) = \widetilde{\theta}_S(\alpha)(I[y+1,N]\star A)$;
    \label{26082113032}
  \item $I[y,N]\star A' = I[y+1,N]\star A'$ ~ (note: $\widetilde{\theta}_S(\alpha)(A)=A'$).
    \label{26082113033}
\end{enumerate}
By \eqref{eq:membership-by-translation}, \ref{26082113031}, and \ref{26082113033}, we obtain
\begin{equation}\label{eq:large-membership-preserved}
  y\in A \text{~if and only if~} y\in A' \text{~for all~} y\geqslant c(S).
\end{equation}

Now consider the element $y\in S$ satisfying $0<y<c(S)$,
since $|\mathbb{N}\backslash S|<\infty$,
we may choose an element $t\in S$ such that $t>\mathfrak{m}$ and $y+t\geqslant c(S)$.
Lemma \ref{lem:two-point-Kleisli-fixed} shows that $\widetilde{\theta}_S(\alpha)(\{0,t\})=\{0,t\}$,
then $\widetilde{\theta}_S(\alpha)(A\star\{0,t\})
= \widetilde{\theta}_S(\alpha)(A)\star\widetilde{\theta}_S(\alpha)(\{0,t\})
= A'\star\{0,t\}$. It follows that
\begin{equation}\label{eq:small-to-large-A}
  y\in A \text{~if and only if~} y+t\in A\star\{0,t\}
\end{equation}
Indeed:
\begin{itemize}
  \item The direction ``$\Rightarrow$'' of \eqref{eq:small-to-large-A} is trivial;
  \item Conversely, if $y+t\in A\star\{0,t\}=A\cup(t+A)$, then $y+t>\mathfrak{m}$, so $y+t\notin A$.
    Thus, $y+t=t+x$ for some $x\in A$, and cancellation gives $x=y$. Thus the direction ``$\Leftarrow$'' of \eqref{eq:small-to-large-A} holds
\end{itemize}
And the same argument gives
\begin{equation}\label{eq:small-to-large-A-prime}
  y\in A' \text{~if and only if~} y+t\in A'\star\{0,t\}.
\end{equation}
Recall that $y+t\geqslant c(S)$, notice that
$\tilde{\theta}_{\fin,S}(\alpha)(A\star\{0,t\})
= \tilde{\theta}_{\fin,S}(\alpha)(A)\star\tilde{\theta}_{\fin,S}(\alpha)(\{0,t\})
= A'\star\{0,t\}$, then \eqref{eq:large-membership-preserved} gives that
$y+t\in A\star\{0,t\}$ if and only if $y+t\in A'\star\{0,t\}$.
Combining it with \eqref{eq:small-to-large-A} and \eqref{eq:small-to-large-A-prime}, we obtain
\[ y\in A \text{~if and only if~} y\in A', \quad y\in S.\]
Note that $0\in A$ and $0\in A'$ here. Furthermore, we obtain $A=A'$. It follows that $\alpha(f_A)=f_{A'}=f_A$.
\end{proof}

\begin{theorem}\label{thm:direct-Kleisli-rigidity}
Let $S$ be a numerical monoid properly contained in $\mathbb N$,
then \[ \Aut(\Hom_{\Kl(\mathscr P_{\fin})}(\mathbb Z/2\mathbb Z,S)) \cong \{1\}. \]
\end{theorem}

\begin{proof}
Take an arbitrary automorphism $\alpha\in\Aut(\Hom_{\Kl(\mathscr P_{\fin})} (\mathbb Z/2\mathbb Z,S) )$ and an arbitrary Kleisli morphism $f \in \Hom_{\Kl(\mathscr P_{\fin})}(\mathbb Z/2\mathbb Z,S)$, then there exists a unique
$A\in M_S$ such that $f=f_A$.
Since $\alpha(f_A)$ is also an element of in same Kleisli Hom-space, there exists a unique $A'\in M_S$ such that $\alpha(f_A)=f_{A'}$.
Then by Proposition \ref{prop:Kleisli-membership-preserved}, we have $y\in A$ if and only if $y\in A'$ ($y\in S$).
Therefore, $A=A'$, and hence $\alpha(f) = \alpha(f_A) = f_{A'} = f_A = f$.
By the arbitrariness of $f$, we have $\alpha = \id_{\Hom_{\Kl(\mathscr P_{\fin})}
    (\mathbb Z/2\mathbb Z,S)}$ as required.
\end{proof}
\begin{remark}\rm
Tringali and Yan conjectured that $\mathcal P_{\fin,0}(S)$ is rigid whenever $S$ is a numerical monoid
properly contained in $\mathbb N$ (\cite[Section 4]{TY2025}). It is an easy consequence of Theorems \ref{thm:Kleisli=Pfin} and \ref{thm:direct-Kleisli-rigidity} that the Tringali--Yan conjecture holds.

\end{remark}
\paragraph{Authors' Contributions}


Regarding the author order: The authors are normally listed alphabetically by surname.
In this submission, however, Haicun Wen is placed as the first author to satisfy the doctoral graduation requirement of Northwest Normal University.
The remaining two authors are ordered alphabetically by surname.
All authors have approved this ordering, and all authors contributed equally to the conception, methodology, derivation, and
writing of this paper.

\paragraph{Competing Interests}

The author declare that they have no conflicts of interest as defined by the journal, nor any other interests that could be perceived as influencing the results presented in this paper.

\paragraph{Data availability}
Data sharing not applicable to this article as no datasets were generated or analysed during the current study

\paragraph{Ethical Approval}

This article does not require ethical approval.

\paragraph{Fundings}
Jian He is supported by the National Natural Science Foundation of China (Grant No. 12501048) and the Hongliu Outstanding Young Talents Funding of Lanzhou University of Technology.
Yu-Zhe Liu is supported by
the National Natural Science Foundation of China (Grant Nos. 12401042 and 12561008),
the Science and Technology Foundation of the Guizhou S\&T Department (Grant Nos. KJLYRC[2026]091, VZD[2026]001, ZD[2025]085 and ZK[2024]YiBan066),
and Scientific Research Foundation of Guizhou University (Grant No. [2023]16).

\addcontentsline{toc}{section}{References}

%

\end{document}